\documentclass[11pt,oneside,reqno]{article}
\usepackage[width=15.7cm,height=24cm,centering]{geometry}
\usepackage{amsmath,amsthm,amssymb,color,bm}
\usepackage[authoryear]{natbib}
\usepackage{booktabs}

\usepackage{tikz}
\RequirePackage{amsthm,amsmath,amsfonts,amssymb,mathrsfs,dsfont,mathtools,thmtools}
\RequirePackage[colorlinks,citecolor=blue,urlcolor=blue,linkcolor=blue]{hyperref}
\usepackage[capitalize]{cleveref}
\RequirePackage{graphicx}

\crefname{equation}{}{}
\crefformat{equation}{\textup{(#2#1#3)}}
\crefrangeformat{equation}{\textup{(#3#1#4)--(#5#2#6)}}
\crefdefaultlabelformat{#2\textup{#1}#3}
\crefname{lemma}{Lemma}{Lemmas}
\crefname{page}{p.}{pp.}
\usepackage[normalem]{ulem}

\usepackage{bbm}
\usepackage{mathrsfs}
\usepackage{graphicx}
\usepackage{tikz}
\usepackage{enumitem}
\usetikzlibrary{arrows, automata}
\usetikzlibrary{calc}
\usetikzlibrary{positioning}

\numberwithin{equation}{section}
\allowdisplaybreaks[4]

\theoremstyle{plain}
\newtheorem{theorem}{Theorem}[section]
\newtheorem{proposition}{Proposition}[section]
\newtheorem{lemma}{Lemma}[section]
\newtheorem{corollary}{Corollary}[section]

\newtheorem{remark}{Remark}[section]

\theoremstyle{definition}

\makeatletter

\newcount\minute
\newcount\hour
\newcount\hourMins
\def\now{%
\minute=\time%
\hour=\time \divide \hour by 60%
\hourMins=\hour \multiply\hourMins by 60%
\advance\minute by -\hourMins%
\zeroPadTwo{\the\hour}:\zeroPadTwo{\the\minute}%
}
\def\zeroPadTwo#1{\ifnum #1<10 0\fi#1}

\renewcommand{\cite}{\citet}

\def\^#1{\ifmmode {\mathaccent"705E #1} \else {\accent94 #1} \fi}
\def\~#1{\ifmmode {\mathaccent"707E #1} \else {\accent"7E #1} \fi}

\def\*#1{#1^\ast}
\edef\-#1{\noexpand\ifmmode {\noexpand\bar{#1}} \noexpand\else \-#1\noexpand\fi}
\def\>#1{\vec{#1}}
\def\.#1{\dot{#1}}

\def\atop{\@@atop}
\def\*#1{\mathscr{#1}}

\renewcommand{\leq}{\leqslant}
\renewcommand{\le}{\leqslant}
\renewcommand{\geq}{\geqslant}
\renewcommand{\ge}{\geqslant}

\newcommand{\Tr}{\mathop{\mathrm{Tr}}}

\newcommand{\N}{\mathcal{N}}

\newcommand{\IE}{\mathbb{E}}
\newcommand{\IP}{\mathbb{P}}
\newcommand{\Var}{\mathop{\mathrm{Var}}\nolimits}
\newcommand{\Cov}{\mathop{\mathrm{Cov}}}

\newcommand{\IR}{\mathbb{R}}

\def\be#1{\begin{equation*}#1\end{equation*}}
\def\ben#1{\begin{equation}#1\end{equation}}
\def\bes#1{\begin{equation*}\begin{split}#1\end{split}\end{equation*}}
\def\besn#1{\begin{equation}\begin{split}#1\end{split}\end{equation}}

\def\mid{\vert}

\def\beqn#1\eeqn{\begin{align}#1\end{align}}
\def\beq#1\eeq{\begin{align*}#1\end{align*}}

\usepackage{graphicx}
\usepackage{latexsym}
\usepackage{amsmath,amsthm,amssymb,amscd}
\usepackage{epsf,amsmath}

\def\E{{\IE}}
\def\P{{\IP}}

\newcommand{\mcl}[1]{\mathcal{#1}}

\renewcommand\section{\@startsection {section}{1}{\z@}%
{-2.8ex \@plus -.8ex \@minus -.2ex}%
{1.0ex \@plus.2ex}%
{\center\small\sc\mathversion{bold}}}

\def\subsection#1{\@startsection {subsection}{2}{0pt}%
{-2.8ex \@plus -.8ex \@minus -.2ex}%
{.8ex \@plus.2ex}%
{\bf\mathversion{bold}}{#1}}

\def\subsubsection#1{\@startsection{subsubsection}{3}{0pt}%
{\medskipamount}%
{-10pt}%
{\normalsize\itshape}{\kern-2.2ex. #1.}}

\def\blfootnote{\xdef\@thefnmark{}\@footnotetext}

\makeatother

\begin{document}

\title{Martingale central limit theorems in $p$-Wasserstein distance}
\author{Xiao Fang$^*$, Yuta Koike$^\dagger$, Zi-Yao Su$^*$}
\date{\it The Chinese University of Hong Kong$^*$, University of Tokyo$^\dagger$} 
\maketitle

\noindent{\bf Abstract:} 
We obtain multivariate martingale central limit theorems in $p$-Wasserstein distance with respect to the $\ell_r$ norm in $\mathbb{R}^d$ for $p\geq 1$ and $r\in [1,\infty]$, which generalize the results for $p=1$ and $r=2$ in the literature. 
As corollaries, we obtain the Yurinskii coupling and Cram\'er-type moderate deviation results. 
We also provide an illustrative application to the stochastic gradient descent algorithm. 
To prove our main results, we combine Lindeberg's swapping argument with a new Gaussian convolution inequality controlling the $p$-Wasserstein distance between a Gaussian convolved with a perturbation and the Gaussian with matching mean and covariance matrix.
The latter is obtained by developing the recent line of research on $p$-Wasserstein bounds. 

\medskip






\section{Introduction}

Martingale central limit theorems (CLTs) are a fundamental tool for statistical inference with dependent, adaptive, or sequentially collected data (\cite{hall1980martingale,li2018applications}).  
Martingale CLTs are equally central in stochastic approximation and machine learning (\cite{polyak1992acceleration}).  
More recently, martingale and related Markov-chain CLTs have been used to study uncertainty quantification for stochastic gradient descent (SGD), temporal-difference learning, two-timescale stochastic approximation, and reinforcement learning algorithms; see, among others, \citet{anastasiou2019normal}, \citet{borkar2021ode}, \citet{hu2024central}, \citet{srikant2025rates}, and \citet{wu2025uncertainty}.  These applications make non-asymptotic distributional error bounds especially important, since finite-time guarantees are needed to assess the accuracy of confidence intervals, hypothesis tests, and uncertainty estimates produced by learning algorithms.

The non-asymptotic theory of martingale CLTs has a long history, but the available results
mostly focus on the 1-dimensional setting for various distributional distances (see, e.g., \cite{bolthausen1982exact,haeusler1988nonuniform,mourrat2013rate,rollin2018quantitative,fan2020wasserstein,dedecker2022rates,guo2024wasserstein})
In contrast, non-asymptotic multivariate martingale CLTs are comparatively sparse.  Existing results include bounds for smooth test functions with an application to averaged SGD (\cite{anastasiou2019normal}),  Yurinskii-type couplings and high-dimensional Gaussian strong approximations for martingales (\cite{cattaneo2025yurinskii}), Berry--Esseen bounds in hyper-rectangle Kolmogorov or convex distances 
(\cite{belloni2018high, wu2026berry}) and bounds in the $1$-Wasserstein distance for martingales and Markov-chain functionals (\cite{srikant2025rates,wu2025uncertainty}).
However, there are no $p$-Wasserstein bounds for $p>1$ available for general multivariate martingale CLTs; see, for example, the discussion in \cite{kong2026finite}. This is due to the technical difficulty that the duality representation of the 1-Wasserstein distance does not extend to general $p$-Wasserstein distances. 
Also, the connection between the Zolotarev and Wasserstein distances used in the one-dimensional results of \cite{dedecker2022rates} and \cite{guo2024wasserstein} is known in multiple dimensions only for $p=r=2$ (see \cite{bolbotowski2026sharp}).

The goal of this paper is to fill this gap by providing a general multivariate martingale CLT in $p$-Wasserstein distance, for arbitrary $p\ge 1$, with transportation cost measured in an $\ell_r$ norm, $r\in [1,\infty]$.
This is not only a natural question in its own right, but also a useful tool for deriving Yurinskii coupling and Cram\'er-type moderate deviation results (see \cref{sec:app}). The techniques for proving $p$-Wasserstein bounds in multivariate CLTs have been recently developed by \cite{LeNoPe15,bonis2020stein,fang2023p}. However, these techniques are not directly applicable to martingales and new ideas are needed.
In this paper, we further develop the above techniques for $p$-Wasserstein bounds and combine them with Lindeberg's swapping argument to handle martingale CLTs. 

In \cref{sec:main}, we first give a key Gaussian convolution inequality (\cref{prop:1}) that compares a Gaussian vector convolved with a small perturbation to the Gaussian distribution with matching covariance. 
This proposition involves much of our technical contribution, including a general upper bound for the $p$-Wasserstein distance with respect to the $\ell_r$ norm (\cref{lem:follmer}), a Bonis-type expansion for martingales (equation~\eqref{eq:bonis}) and an error bound for comparing two conditional expectations given slightly different conditions (\cref{lem:2}).
We then combine this key proposition with Lindeberg's swapping argument to obtain martingale CLTs in \cref{thm:3,cor:2,cor:1}, assuming first that the conditional covariance matrices add up to a constant matrix. 
\cref{thm:3} offers a natural extension of the 1-Wasserstein bound by \cite[Theorem~3.3]{wu2025uncertainty} to the stronger
\(\mcl W_{p,r}\) distance for general \(p\ge1\) and \(r\in[1,\infty]\).
Finally, in \cref{thm:2,thm:4}, we relax this assumption and obtain bounds involving the difference between the sum of conditional covariances and its expectation.

Our main results exhibit several features, each of which suggests a natural direction for further improvement. First, obtaining the $p$-Wasserstein rates established here generally requires finite $3p$-th moments of the martingale differences. Second, in contrast with the i.i.d.\ setting, the general convergence rate is of order $n^{-1/6}$ unless additional assumptions on the conditional covariances are imposed.
Third, under the finite $3p$-th moment assumptions considered here, our bounds also require $L^{3p/2}$ control of the fluctuations of the predictable quadratic variation around its mean.
Finally, the dimension dependence contains a multiplicative factor of order $d$, irrespective of the choice of $r$ in the $\mcl{W}_{p,r}$ bound. It would be interesting to determine whether any of these moment requirements, convergence rates, or dimension-dependent factors can be improved.


In \cref{sec:app}, we first apply our main results to the Yurinskii coupling for martingales and compare our result with that in \cite{cattaneo2025yurinskii}. We then deduce Cram\'er-type moderate deviation results from our $p$-Wasserstein bounds following the approach of \cite{fang2023p}. The result provides an extension of those in \cite{fan2024cramer} from one dimension to multiple dimensions. Finally, we provide an illustrative application to the SGD algorithm. Our result is potentially applicable to general stochastic approximation problems as mentioned in \cite{kong2026finite} and we leave such an application to future investigation.


In \cref{sec:proof}, we give the proof of our main results, leaving the proof of some technical lemmas to \cref{sec:lem}. The proofs of the applications are given in \cref{sec:proofapp}.

\paragraph{Concurrent work.} In a concurrent work, Morgane Austern (personal communication) and her coauthors considered $p$-Wasserstein bounds with respect to a norm different from the $\ell_r$ norm for random variables of a different structure from martingales. A common feature is that both of our works begin with a bound like in \cref{lem:follmer}.

During the preparation of this work, we became aware of the work of \cite{zhang2026gaussian}, which was posted on arXiv a few days ago. They obtained $p$-Wasserstein bounds with respect to the $\ell_2$ norm in $\mathbb{R}^d$ for $p\geq 2$ in the CLT for additive functionals of Markov chains and the associated martingales constructed by solving the Poisson equation. Their method is tailored to that setting and differs from ours.

\paragraph{Statement of AI use.} GPT 5.6 Pro was used solely to improve the exposition of the manuscript.

\paragraph{Acknowledgements.} We thank Jia-Xue Wang for helpful discussions.
Fang X. was partially supported by Hong Kong RGC GRF 14304822, 14303423, 14302124, 14304125 and a CUHK direct grant. 
Koike Y. was partially supported by JST CREST Grant Number JPMJCR2115 and JSPS KAKENHI Grant Numbers JP24K14848, JP26K02870.

\section{Main results}\label{sec:main}

\paragraph{Notations.}
We adopt the following notation conventions. 
If $x=(x_1,\dots, x_d)^\top$ is a (possibly random) vector, we write $|x|_r:=(|x_1|^r+\dots+|x_d|^r)^{1/r}$ for its $\ell^r$-norm, where $r \in [1,\infty)$ and $|x|_\infty:=\max_{1\leq j\leq d}|x_j|$.

If $A$ is a matrix, then $\|A\|_{r_1\to r_2}:=\sup_{|u|_{r_1}\leq 1}|Au|_{r_2}$ denotes the induced $\ell^{r_1}$--$\ell^{r_2}$ operator norm for $r_1, r_2\in [1,\infty]$.
Write $\|A\|_{r}=\|A\|_{r\to r}$.

If $X$ is a random vector and $A$ is a random matrix defined on the probability space $(\Omega, \mcl{F}, \P)$, we define
\[
\|X\|_{p,r}:=\bigl(\mathbb E|X|_r^p\bigr)^{1/p},
\qquad
\|A\|_{p,r}:=\bigl(\mathbb E\|A\|_r^p\bigr)^{1/p}.
\]
Thus, for a random vector, we first take the $\ell^r$-norm and then the $L^p(\P)$ norm.
Sometimes, we also write in short $|X|=|X|_2$ for a (random) vector and $\|X\|_p=\|X\|_{p,2}$ for a (random) vector.


Let $I_d$ denote the $d\times d$ identity matrix.

In addition, for two probability measures $\mu$ and $\nu$ on $\mathbb{R}^d$ with finite $p$-th moments, and for $r \in [1,\infty]$, the $p$-Wasserstein distance with respect to the $\ell^r$ norm is defined by
\[
\mcl W_{p,r}(\mu,\nu) 
:= \inf_{\pi \in \Pi(\mu,\nu)} 
\Biggl( 
\int_{\mathbb{R}^d \times \mathbb{R}^d} 
|x-y|_r^p \, d\pi(x,y) 
\Biggr)^{1/p},
\]
where $\Pi(\mu,\nu)$ denotes the collection of all couplings of $\mu$ and $\nu$. 
In particular, when $r=2$, this reduces to the standard $p$-Wasserstein distance.

We use $\phi(\cdot)$ to denote the $d$-dimensional standard normal density function.

We assume that the conditional expectations and moments appearing throughout the paper are well-defined and finite.

\medskip
\noindent

\subsection{The key proposition}

In this subsection, we provide a new Gaussian convolution inequality controlling the $p$-Wasserstein distance between a Gaussian convolved with a perturbation and the Gaussian with matching mean and covariance matrix. 
It is obtained by developing further the recent line of research on $p$-Wasserstein bounds (see the proof in \cref{sec:proof}).

\begin{proposition}\label{prop:1}
Let $\xi\in \IR^d$ be a random vector with $\E \xi=0$ and $\Cov(\xi)=V$.
Assume $\Sigma\succ0$. Let $G\in \IR^d$ be a Gaussian random vector with $\E G=0$, $\Cov(G)=\Sigma-V\succeq0$ and independent of $\xi$. 
Let $W=\xi+G$. Then,
for all \( p \ge 1 \) and \( r \in [1,\infty] \),  
the $(p,r)$–Wasserstein distance satisfies
\[
\mcl W_{p,r}(\mcl{L}(W),N(0,\Sigma))
\le Cp\|\xi\|_{3p,r}\|\Sigma^{-1/2}\xi\|_{3p,2}^2,
\]
where $C$ is a universal constant.
\end{proposition}

\begin{remark}
(a) \cite{paulin2026theoretical} recently obtained a related bound on
\be{
\mcl W_{p,2}(\mathcal L(W),N(0,\Sigma-V))
}
by means of Gaussian-convolution
coupling inequalities, and used it to derive Wasserstein bias bounds for
stochastic-gradient kinetic Langevin/MCMC algorithms. 
In contrast,
we compare \(\mathcal L(W)\) with the Gaussian \(N(0,\Sigma)\), whose variance matches
that of \(W\). This removes the leading variance-mismatch term and yields a
smaller-order remainder. This smaller-order bound is crucial in the Lindeberg
argument used below to obtain a CLT.

(b) Lemma 1.6 in \cite{zhai2018high} gives a bound on $\mcl W_{2,2}(\mathcal L(W),N(0,\Sigma))$ when $\xi$ is bounded and $V$ is proportional to $\Sigma$. Apart from these additional conditions, the proof of this lemma relies on a variant of Talagrand's transportation inequality and the rotational invariance of the Euclidean norm. Consequently, extending this proof technique beyond the case $p=r=2$ appears challenging.
\end{remark}

\subsection{Martingale CLTs}

Let \( \xi_1, \dots, \xi_n \in \mathbb{R}^d \) be a martingale difference sequence adapted to a filtration \( \mathcal{F}_0, \dots, \mathcal{F}_n \). In particular,
\[
\mathbb{E} \left[ \xi_k \mid \mathcal{F}_{k-1} \right] = 0 \quad \text{a.s. for } 1 \leq k \leq n.
\]
Define the partial sums (assume that the $\xi$'s are normalized so that $S_n$ is of order $\Theta_p(1)$)
\[
S_0 = 0, \quad S_k = \xi_1 + \cdots + \xi_k, \quad \text{for } 1 \leq k \leq n.
\]
For \( 1 \leq k \leq n \), define the conditional and unconditional covariances
\bes{
&V_k = \mathbb{E} \left[ \xi_k \xi_k^{\top} \mid \mathcal{F}_{k-1} \right], \quad \Pi_k=\sum_{i=k}^n V_i, \quad \Pi_{n+1}=0, \quad \Sigma=\Pi_1,\\
& 
\bar{V}_k = \mathbb{E} \left[ \xi_k \xi_k^{\top} \right] = \mathbb{E} V_k, \quad 
\bar{\Sigma} = \operatorname{Var}(S_n) =\sum_{k=1}^n \bar{V}_k.
}

We begin with our first general result assuming $\Sigma=\bar \Sigma$ a.s. If $\Sigma\ne\bar \Sigma$, additional error terms necessarily appear involving their difference. See \cref{thm:2,thm:tail-cor22,thm:4}.

\begin{theorem}[The case $\Sigma=\bar \Sigma$ a.s.]\label{thm:3}
In the above setting, assume that $\Sigma=\bar \Sigma$ a.s. Let $M\succeq0$ be a deterministic symmetric $d\times d$ matrix such that $\Pi_k+M\succ0$ a.s. for $1\le k\le n$ (we will take $M=0$ in \cref{cor:2}, and $M=\sigma^2 I_d$ for some $\sigma^2>0$ in \cref{cor:1}).
We have
\besn{\label{eq:thm3}
&\mcl W_{p,r}(\mcl{L}(S_n), N(0, \bar\Sigma))\\
\leq& C p \sum_{k=1}^n
\left\|
\lVert\xi_k\,|\,\mcl{F}_{k-1}\rVert_{3p,r}
\lVert(\Pi_k+M)^{-1/2} \xi_k\,|\,\mcl{F}_{k-1}\rVert_{3p,2}^{2}
\right\|_{p}
+2 \|M^{1/2} Z\|_{p,r},
}
where $C$ is a universal constant and $Z\sim N(0,I_d)$. 
\end{theorem}

\begin{remark}
(a) \cite[Theorem~3.3]{wu2025uncertainty}
obtained a 1-Wasserstein bound for vector-valued martingales under the deterministic
terminal quadratic variation condition. For $M\succ0$, translating their notation into ours, their bound reads as
\[
\begin{aligned}
\mcl W_{1,2}(\mcl{L}(S_n), N(0, \bar\Sigma))
\leq &
C\Big\{ \bigl(2+\log(dn\|\bar\Sigma+M\|_2)\bigr)^+
\sum_{k=1}^n
\E\!\left[
    |\xi_k|_2 |(\Pi_k+M)^{-1/2}\xi_k|_2^2 
\right]
\\
&\quad
+ \frac{1}{\sqrt n}\,
    \Tr\!\left\{\log(\bar\Sigma+M)-\log M\right\}
+ \sqrt{\Tr(M)} \Big\}.
\end{aligned}
\]
\cref{thm:3} offers a natural extension to
\(\mcl W_{p,r}\) distances for general \(p\ge1\) and \(r\in[1,\infty]\). 
Moreover, the linear growth in $p$ in \cref{eq:thm3} is optimal as will be argued in \cref{rem:optimalp}.

(b) For sums of independent random vectors, the optimal convergence rate in the
\(p\)-Wasserstein CLT can be obtained under a finite \((p+2)\)-moment
assumption on the summands (\cite{bonis2020stein}). This moment order
comes from the Rosenthal-type inequality used in the independent case. In the
martingale setting considered here, our argument instead relies on
\cref{prop:1}, whose one-step error involves \(3p\)-th moments of the increments.
Consequently, our bound requires a finite \(3p\)-th moment assumption. See a related \(4p\)-th moment condition required in the exchangeable-pair
approach of \cite{fang2023p}.
\end{remark}

To apply \cref{eq:thm3} and other results in the following, we need to bound $\|X\|_{p,r}$ for \(X\sim N(0,\Sigma)\). Such a bound was obtained by \cite[Lemma B.4]{cattaneo2025yurinskii} for $p=1$. 
The lemma below extends their result for all \(p\ge 1\). Its proof is deferred to \cref{sec:lem}.

\begin{lemma}\label{lem:gaussian-norm-bound}
Let \(X\sim N(0,\Sigma)\), where \(\Sigma\in\mathbb R^{d\times d}\) is symmetric positive
semidefinite. 
Then, for \(1\le r<\infty\),
\[
\begin{gathered}
\|X\|_{p,r}
\le C\sqrt r\left(\sum_{j=1}^d\Sigma_{jj}^{r/2}\right)^{1/r}
+C\sqrt p\,\|\Sigma^{1/2}\|_{2\to r},\\[-1pt]
\|X\|_{p,\infty}
\le C\sqrt{\max_{1\le j\le d}\Sigma_{jj}}\,
\bigl(\sqrt{\log(2d)}+\sqrt p\bigr),
\end{gathered}
\]
where $C$ is a universal constant.
\end{lemma}

To demonstrate the typical order of the bound \cref{eq:thm3}, we give two corollaries below. In the first corollary, we assume a uniform lower bound on the conditional covariances $V_k$. Then, under suitable moment assumptions, we obtain a $\log n/\sqrt{n}$ rate.

\begin{corollary}[Nearly $n^{-1/2}$ rate under uniform ellipticity]\label{cor:2}
Assume the setting of \cref{thm:3}, and assume in addition that
$\bar \Sigma \succ 0$ and, for some constant $0<\alpha\le1$,
\[
        \frac{\alpha}{n}\bar \Sigma \preceq V_k,
        \qquad 1\le k\le n,
\]
almost surely. 
Then, for all \(p\ge 1\) and
\(r\in[1,\infty]\),
\ben{\label{eq:cor2.1-1}
\mcl W_{p,r}\bigl(\mathcal L(S_n),N(0,\bar \Sigma)\bigr)
\le \frac{Cp}{\alpha}
\sum_{k=1}^n
\frac{n}{n-k+1}
\left\|
   \|\xi_k\mid\mathcal F_{k-1}\|_{3p,r}
   \|\bar \Sigma^{-1/2}\xi_k\mid\mathcal F_{k-1}\|_{3p,2}^2
\right\|_{p},
}
where \(C\) is a universal constant.

In particular, if (recall that the $\xi$'s are normalized)
\ben{\label{eq:cor2.1cond}
   \|\xi_k\|_{3p,r}
   \|\bar \Sigma^{-1/2}\xi_k\|_{3p,2}^2
\leq \frac{K}{n^{3/2}}
}
for a constant $K$,
then
\ben{\label{eq:cor2.1-2}
\mcl W_{p,r}\bigl(\mathcal L(S_n),N(0,\Sigma)\bigr)
\le
\frac{CKp}{\alpha}\cdot\frac{\log(en)}{\sqrt n}.
}
\end{corollary}

\begin{remark}\label{rem:optimalp}
(a) Under the assumptions of \cref{cor:2}, if \(p,r,d\) are fixed and the condition \cref{eq:cor2.1cond} is satisfied, then
\cref{cor:2} yields
\[
    \mcl W_{p,r}\bigl(\mathcal L(S_n),N(0,\bar \Sigma)\bigr)
    = O\!\left(\frac{\log n}{\sqrt n}\right).
\]
Thus we recover the usual \(n^{-1/2}\) Berry--Esseen rate up to a logarithmic
factor. The linear dependence on \(p\) in our bound is also 
optimal as discussed by \cite{fang2023p}.

(b) A near-$n^{-1/2}$ rate has also been obtained for $1$-Wasserstein distance under a boundedness condition (\cite[Corollary~3.4]{wu2025uncertainty}) or an expected uniform ellipticity assumption (\cite[Theorem~1]{srikant2025rates}). However, their arguments work directly with
Lipschitz test functions, or with the corresponding Stein/Poisson equation, and
exploit cancellations at the level of expectations. Such a test-function
representation is not available for \(\mcl W_{p,r}\) when \(p>1\), where we instead use the coupling and Gaussian
convolution approach.
\end{remark}

In the general case, however, the conditional covariances $V_k$ can have wild behavior and we obtain a bound that is typically of the order $O(n^{-1/6})$.

\begin{corollary}[$n^{-1/6}$-rate bound]\label{cor:1}
Assume the setting of \cref{thm:3}; in particular, $\Sigma=\bar\Sigma$ a.s.  
Let
\[
A_{p,r}:=\sum_{k=1}^n
\left\|\|\xi_k\mid\mathcal F_{k-1}\|_{3p,r}
\|\xi_k\mid\mathcal F_{k-1}\|_{3p,2}^{2}\right\|_p.
\]
Then, for every $p\ge1$ and $r\in[1,\infty]$, we have
\ben{\label{eq:cor1}
\mcl W_{p,r}\!\left( \mathcal L(S_n),N(0,\bar\Sigma)\right)
\le
C(pA_{p,r})^{1/3}\|Z\|_{p,r}^{2/3},
}
where $C$ is a universal constant and $Z\sim N(0,I_d)$.
\end{corollary}

\begin{remark}\label{rem:typical16}
In typical situations, because of the normalization, we expect $\xi_k\sim 1/\sqrt{n}$, hence $A_{p,r}\sim 1/\sqrt{n}$. In such situations, the bound in \cref{eq:cor1} is of the order $O(n^{-1/6})$.
\end{remark}

Finally, we consider the general case in which $\Sigma\ne \bar \Sigma$. Then, additional error terms involving $\Sigma-\bar \Sigma $ necessarily appear.
In the next result, we give a simple bound in the situation in which $\Sigma-\bar \Sigma\preceq \Delta_0$ for a fixed $d\times d$ positive semidefinite matrix $\Delta_0$. 

\begin{theorem}[The case \(\Sigma-\bar\Sigma\preceq\Delta_0\)]
\label{thm:2}
Under the martingale setup above, assume that
\[
    \Sigma-\bar\Sigma\preceq \Delta_0
    \qquad\text{almost surely},
\]
where \(\Delta_0\) is a deterministic symmetric positive semidefinite matrix.
Put
\[
    H:
      =\Delta_0-(\Sigma-\bar\Sigma).
\]
Then \(H\succeq0\) almost surely. For any deterministic symmetric
\(d\times d\) matrix \(M\succeq0\) such that \(\Pi_k+M\succ0\) a.s.
for \(1\le k\le n\), we have
\[
\begin{aligned}
\mcl W_{p,r}\bigl(\mathcal L(S_n),N(0,\bar\Sigma)\bigr)
\le\;&
Cp\sum_{k=1}^n
\Bigl\|
    \|\xi_k\mid\mathcal F_{k-1}\|_{3p,r}
    \|(\Pi_k+M)^{-1/2}\xi_k
        \mid\mathcal F_{k-1}\|_{3p,2}^{2}
\Bigr\|_p
\\
&\quad
+2\|M^{1/2}Z\|_{p,r}
+\|H^{1/2} Z\|_{p,r}
+\|\Delta_0^{1/2}Z\|_{p,r},
\end{aligned}
\]
where \(C\) is a universal constant and \(Z\sim N(0,I_d)\) is independent of everything else.
\end{theorem}

Note that $\|H^{1/2} Z\|_{p,r}$ can be bounded using \cref{lem:gaussian-norm-bound} conditionally on $H$.

If there is no almost sure upper bound for $\Sigma-\bar \Sigma$, but we are interested in comparing probabilities, then the error bound will involve $\P(\Sigma-\bar \Sigma\npreceq \Delta_0)$; see the proof of \cref{thm:tail-cor22} in the application to the Yurinskii coupling.

In general, we obtain a bound involving the moments of $\Sigma-\bar \Sigma$ as in the next result.

\begin{theorem}[The case \(\Sigma\neq\bar\Sigma\): two explicit forms]
\label{thm:4}
Let \(p\ge1\), \(r\in[1,\infty]\), and let \(Z\sim N(0,I_d)\) be independent.
Put
\[
 R_{p,r}:=\|S_n\|_{3p,r}+\|\bar\Sigma^{1/2}Z\|_{p,r}.
\]
Then the following bounds hold, with a universal constant \(C\).

\smallskip
\noindent
\textup{(i)} \emph{Elliptic form, corresponding to \cref{cor:2}.}
Assume \(\bar\Sigma\succ0\) and, for some \(0<\alpha\le1\),
\(V_k\succeq(\alpha/n)\bar\Sigma\) almost surely for \(1\le k\le n\).
Then
\begin{equation}\label{eq:thm23-elliptic}
\begin{aligned}
 &\mcl W_{p,r}\bigl(\mcl L(S_n),N(0,\bar\Sigma)\bigr)\\
 &\le\frac{Cp}{\alpha}\sum_{k=1}^n\frac n{n-k+1}
 \left\|\|\xi_k\mid\mathcal F_{k-1}\|_{3p,r}
 \|\bar\Sigma^{-1/2}\xi_k\mid\mathcal F_{k-1}\|_{3p,2}^2\right\|_p\\
 &\quad+C\|Z\|_{p,r}^{2/3}
 \left(R_{p,r}\|\Sigma-\bar\Sigma\|_{3p/2,2}\right)^{1/3}.
\end{aligned}
\end{equation}

\smallskip
\noindent
\textup{(ii)} \emph{Smoothing form, corresponding to \cref{cor:1}.}
Without an ellipticity assumption, and allowing singular \(\bar\Sigma\),
with \(A_{p,r}\) as in \cref{cor:1},
\begin{equation}\label{eq:thm23-smoothing}
 \mcl W_{p,r}\bigl(\mcl L(S_n),N(0,\bar\Sigma)\bigr)
 \le C\|Z\|_{p,r}^{2/3}
 \left(pA_{p,r}+R_{p,r}\|\Sigma-\bar\Sigma\|_{3p/2,2}\right)^{1/3}.
\end{equation}
\end{theorem}

In applying \cref{thm:4}, $\|S_n\|_{3p,r}$ may be bounded using the standard Rosenthal--Burkholder inequality. We record such an inequality in \cref{lem:4}.

\section{Applications}\label{sec:app}

In this section, we deduce the Yurinskii coupling and Cram\'er-type moderate deviation results from the $p$-Wasserstein bounds from \cref{sec:main}. We also provide an illustrative application to the stochastic gradient descent algorithm. The proofs of the applications are given in \cref{sec:proofapp}.

\subsection{The Yurinskii coupling}\label{sec:yurinskii}

Yurinskii’s coupling (\cite{yurinskii1978error}) has proven to be an important theoretical tool for developing
non-asymptotic distributional approximations in mathematical statistics and applied probability.
Given a sequence of random vectors $S_n$, the Yurinskii coupling takes the following form:
for each $n\geq 1$, $r\in [1,\infty]$ and $\eta>0$, there exists a Gaussian random vector $T_n(\eta)$ with 
\ben{\label{eq:defYurinskii}
\P(|S_n-T_n(\eta)|_r>\eta)\leq r_n(\eta),
}
where
$r_n(\eta)$ is a vanishing approximation error in certain asymptotic region.
Most of the literature on the Yurinskii coupling deals with the case in which $S_n$ is a sum of independent random vectors.
\cite{li2020uniform} and \cite{cattaneo2025yurinskii} considered the case in which $S_n$ is a martingale.
We refer to \cite{cattaneo2025yurinskii} for more references on the Yurinskii coupling.

As remarked in \cite[Remark~2.1]{cattaneo2025yurinskii}, each coupling variable $T_n(\eta)$ in \cref{eq:defYurinskii} depends on $\eta$, and, as such, \cref{eq:defYurinskii} cannot be used to deduce a bound on $\|S_n-T_n\|_{p,r}$. 
In contrast, one may deduce \cref{eq:defYurinskii} from a Wasserstein bound and Markov's inequality.
We obtain the following consequence of our $p$-Wasserstein bounds.

\begin{theorem}
\label{thm:tail-cor22}
Retain the martingale setup and notation of Section~2.2, and do not assume
$\Sigma=\bar\Sigma$ a.s. Put
\[
        \Delta:=\Sigma-\bar\Sigma .
\]
Let $A_{p,r}$ be as in Corollary~\ref{cor:1}. 
Then, for every $p\ge1$, $r\in[1,\infty]$, and $\eta>0$, there exists, on an
extension of the original probability space, a random vector
$T:=T(\eta)\sim N(0,\bar\Sigma)$ such that
\besn{\label{eq:thm3.1}
\mathbb P\left(|S_n-T|_r>3\eta\right)
&\le
\left[
\frac{
C(pA_{p,r})^{1/3}\|Z\|_{p,r}^{2/3}
}{\eta}
\right]^p                                                        \\
&\quad+
C\left(
\frac{\mathbb E\|\Delta\|_2\,\|Z\|_{1,r}^{2}}{\eta^2}
\right)^{1/3},
}
where $Z\sim N(0,I_d)$ and $C$ is a universal constant.
\end{theorem}

\begin{remark}
As remarked in \cref{rem:typical16}, the first term on the right-hand side of \cref{eq:thm3.1} is of order $\sim (n^{-1/6}\eta^{-1})^p$ under suitable moment conditions, while \cite[Proposition~2.1]{cattaneo2025yurinskii} had the bound corresponding to $p=1$. Therefore, our error bound may be smaller under higher-moment assumptions.
\end{remark}

\subsection{Cram\'er-type moderate deviations}

Moderate deviations date back to
\cite{Cramer1938} who obtained expansions for tail probabilities for sums of independent
 random variables about the normal distribution. For independent and identically distributed (i.i.d.) random
 variables $X_1, \cdots, X_n$ with $\E X_1=0$ and $\Var(X_1)= 1$ such that $\E e^{|X_1|/b}\leq C < \infty$ for some $b>0$, it follows from \cite[Chapter~8, Eq.~(2.41)]{petrov2012sums} that
 \be{
 \left| \frac{\P(W>x)}{\P(Z>x)}-1\right|=O(1) (1+x^3)/\sqrt{n}
}
 for $0\leq x\leq O(1) n^{1/6}$, where $W=(X_1+\cdots +X_n)/ \sqrt{n}$, $Z\sim N(0,1)$ and $O(1)$ is bounded by a constant that depends on $b$ and $C$. 
 The range $0\leq x\leq O(1) n^{1/6}$ and the order of the error term $O(1) (1+x^3)/\sqrt{n}$ are optimal. \cite{von1967multi} obtained a multi-dimensional generalization of the result of \cite{Cramer1938} for sums of independent random vectors.
 In contrast, while one-dimensional moderate deviation results have been actively studied for the martingale CLT (see \cite{fan2024cramer} and the references therein), multi-dimensional results have not been available. Using our results in \cref{sec:main}, combined with the approach of \cite{fang2023p}, we fill in this gap in the literature. In particular, using \cref{cor:2} and \cref{thm:2}, we obtain the following result.


\begin{theorem}\label{thm:MDmartingale}
Retain the martingale setup and notation of \cref{sec:main}, and assume $\bar \Sigma=I_d$, $|\xi_k|\leq K/\sqrt{n}$, $-\frac{\gamma^2}{n} I_d\preceq \Sigma-\bar \Sigma\preceq \frac{\gamma^2}{n} I_d$, and, for some constant $0<\alpha\le1$,
\[
        \frac{\alpha}{n}\Sigma \preceq V_k,
        \qquad 1\le k\le n,
\]
almost surely. Then, we have, with $\delta=\log(en)/\sqrt{n}$,
\ben{\label{eq:thmMD}
\left|\frac{\P(|S_n|>x)}{\P(|Z|>x)}-1\right|\leq C(1+x)(1+|\log\delta|+x^2)\delta
}
for all $0\leq x\leq \delta^{-1/3}$,
where \(Z\sim N(0,I_d)\) and $C$ is a constant depending only on $K$, $\alpha$, $\gamma$ and $d$.

\end{theorem}

\begin{remark}
By comparison with the one-dimensional result, the scale $x=o(\delta^{-1/3})$ on which the relative error on the left-hand side of \cref{eq:thmMD} vanishes is optimal up to the $\log n$ factor.
The conditions in \cref{thm:MDmartingale} can be relaxed using other results in \cref{sec:main}. However, the convergence rate will be worse.
\end{remark}

\subsection{Stochastic approximation}

Recently, martingale CLTs with nonasymptotic error bounds have been actively applied to stochastic approximation algorithms; see \cite{anastasiou2019normal}, \cite{srikant2025rates}, \cite{wu2025uncertainty}, and \cite{kong2026finite}.

As a classical case, we consider the Polyak--Ruppert averaged stochastic gradient descent (SGD) algorithm. Fix $p\geq1$. Throughout this subsection, we work with the Euclidean norm ($r=2$). Let $f:\mathbb R^d\to\mathbb R$ be twice continuously differentiable, let $(\mathcal F_k)_{k\geq0}$ be a filtration, and assume that $\theta_0$ is $\mathcal F_0$-measurable. Consider the recursion
\[
    \theta_k
    =
    \theta_{k-1}
    -
    \eta_k\{\nabla f(\theta_{k-1})+\varepsilon_k\},
    \qquad k\geq1,
\]
where $\theta_k\in\mathbb R^d$, $\eta_k=\eta_0k^{-\alpha}$ with
$\eta_0>0$ and $\alpha\in(1/2,1)$, and each $\varepsilon_k$ is
$\mathcal F_k$-measurable with
\[
    \mathbb E(\varepsilon_k\mid\mathcal F_{k-1})=0.
\]
Let $\theta^*$ be the unique minimizer of $f$, and set
\[
    \Delta_k=\theta_k-\theta^*,
    \qquad
    \bar\theta_n=\frac1n\sum_{k=0}^{n-1}\theta_k,
    \qquad
    H=\nabla^2f(\theta^*).
\]

We assume that, for some constants $0<\mu\leq L<\infty$ and $L_H<\infty$,
\begin{equation}\label{eq:sgd-regularity}
    \mu I_d
    \preceq
    \nabla^2f(\theta)
    \preceq
    LI_d,
    \qquad
    \|\nabla^2f(\theta)-\nabla^2f(\theta')\|_2
    \leq
    L_H|\theta-\theta'|_2
\end{equation}
for all $\theta,\theta'\in\mathbb R^d$, and that $\eta_0L\leq1/2$ and
$\|\Delta_0\|_{2p}<\infty$. The noise satisfies
\begin{equation}\label{eq:sgd-moment}
    \|\varepsilon_k\|_{3p,2}\leq K_{3p}<\infty,\qquad k\geq1.
\end{equation}
Assume in addition that, for a deterministic symmetric matrix $V\succ0$,
\begin{equation}\label{eq:sgd-fixed-covariance}
    \mathbb E[\varepsilon_k\varepsilon_k^\top\mid\mathcal F_{k-1}]
    =V\quad\text{a.s.},\qquad k\geq1.
\end{equation}

As an application of \cref{thm:3}, we obtain the following result.
\begin{proposition}\label{prop:SGD}
Under the above assumptions, for every $n\geq2$,
\begin{equation}\label{eq:sgd-limit-bound}
\begin{aligned}
 &\mcl W_{p,2}\!\left(
 \mcl L\bigl(\sqrt n(\bar\theta_n-\theta^*)\bigr),
 \N(0,H^{-1}VH^{-1})\right)\\
 &\qquad\leq C\left\{
 \frac{\log(en)}{\sqrt n}
 +n^{1/2-\alpha}+n^{\alpha-1}\right\}.
\end{aligned}
\end{equation}
Here $C$ depends only on $p,d,\alpha,\eta_0,\mu,L,L_H,K_{3p}$,
$\|V^{-1}\|_2$, and $\|\Delta_0\|_{2p}$.
In particular, the bound is $Cn^{-1/4}$ when $\alpha=3/4$.
\end{proposition}

\begin{remark}
Under related assumptions, \cite{anastasiou2019normal} obtain
smooth-test-function bounds for averaged SGD, whereas
\cite[Theorem~3.4]{shao2022berry} and
\cite[Theorems~2 and~4]{sheshukova2026gaussian} give convex-distance
approximations; the latter work also obtains an $n^{-1/4}$ rate for the limiting-covariance target at step-size exponent $3/4$.
We instead obtains a bound on $\mcl W_{p,2}$.
These results are complementary.
\end{remark}

\section{Proofs of the main results}\label{sec:proof}

In this section, we prove the main results stated in \cref{sec:main}. We defer some lemmas and their proofs to \cref{sec:lem}. In all the proofs, we use $C$ to denote universal constants, whose value may differ in different expressions. Also, we assume that the moments appearing below are all finite; otherwise, the desired bounds are trivial.

\begin{proof}[Proof of \cref{prop:1}]
We first prove the case $\Sigma-V\succ0$. Put
\[
V_1:=V,\qquad V_2:=\Sigma-V,\qquad \xi_\Sigma:=\Sigma^{-1/2}\xi.
\]
On an extension of the probability space, write
\(G=V_2^{1/2}Z_1\), where \(Z_1\sim N(0,I_d)\), and let
\(Z_2\sim N(0,I_d)\) be independent of \((\xi,Z_1)\).

We begin by applying \cref{lem:follmer}, which generalizes the corresponding result in the literature by allowing $\Sigma\ne I_d$ and $r\ne 2$:
\ben{\label{eq:masterbound}
\mcl W_{p,r}\bigl(\mcl L(W),N(0,\Sigma)\bigr)
\leq 2\int_0^{\infty} \left\|\rho_t\!\left(F_t\right)\right\|_{p,r} \, dt,
\quad p \geq 1,
}
where
\[
F_t = e^{-t} \Sigma^{-1/2}W + \sqrt{1 - e^{-2t}} Z_2,
\quad
\rho_t(F_t) = e^{-t} \mathbb{E} \left[\left. W - \frac{1}{\sqrt{e^{2t}-1}} \Sigma^{1/2} Z_2 \,\right|\, F_t\right].
\]
Here $F_t$ is the image under $\Sigma^{-1/2}$ of the conditioning vector in
\cref{lem:follmer}. This invertible deterministic transformation preserves
its sigma-field and leaves the vector inside the conditional expectation,
and hence the $\ell_r$ cost in \cref{eq:masterbound}, unchanged.
Let
\[
V_{1,\Sigma}:= \Sigma^{-1/2}V_1 \Sigma^{-1/2},
\qquad
\widetilde I_{d,t} := I_d - e^{-2t}V_{1,\Sigma},
\qquad
c_t:=1-e^{-2t}\|V_{1,\Sigma}\|_2.
\]
Since
\(0\preceq V_{1,\Sigma}\prec I_d\), for every $t>0$,
\begin{equation}\label{eq:Idbound}
\widetilde I_{d,t}\succeq c_t I_d\succeq(1-e^{-2t})I_d,
\qquad
\|\widetilde I_{d,t}^{-1/2}\|_2\le c_t^{-1/2}.
\end{equation}

We may write
\[
\begin{aligned}
F_t
&= e^{-t}\Sigma^{-1/2}\xi
   +e^{-t}\Sigma^{-1/2}V_2^{1/2}Z_1
   +\sqrt{1-e^{-2t}}\,Z_2 \\
&= e^{-t}\Sigma^{-1/2}\xi+\widetilde I_{d,t}^{1/2}Z,
\end{aligned}
\]
where
\[
Z:=\widetilde I_{d,t}^{-1/2}
\left(
e^{-t}\Sigma^{-1/2}V_2^{1/2}Z_1
+\sqrt{1-e^{-2t}}\,Z_2
\right)\sim N(0,I_d).
\]
For each fixed $t>0$, this $Z$ is standard Gaussian and independent of $\xi$;
its realization may depend on $t$. By Gaussian regression,
\[
\mathbb E\!\left[
    V_2^{1/2}Z_1-\frac{1}{\sqrt{e^{2t}-1}}\Sigma^{1/2}Z_2
    \,\middle|\,\xi,Z
\right]
=
-e^{-t}V_1\Sigma^{-1/2}\widetilde I_{d,t}^{-1/2}Z .
\]
Hence, since $\widetilde I_{d,t}^{-1/2}$ is invertible,
\[
\rho_t(F_t)
= e^{-t}\mathbb E\!\left[
    \xi-e^{-t}V_1\Sigma^{-1/2}\widetilde I_{d,t}^{-1/2}Z
    \,\middle|\,
    e^{-t}\widetilde I_{d,t}^{-1/2}\xi_\Sigma+Z
\right].
\]
We first consider $\|\xi_\Sigma\|_{3p,2}^2>1/2$.
Since $V_1\Sigma^{-1/2}=\mathbb E[\xi\xi_\Sigma^\top]$, Minkowski's
inequality, the one-dimensional Gaussian moment bound, and H\"older's
inequality give
\[
\begin{aligned}
\|V_1\Sigma^{-1/2}\widetilde I_{d,t}^{-1/2}Z\|_{p,r}
&\le C\sqrt p\,\mathbb E\!\left[
|\xi|_r|\widetilde I_{d,t}^{-1/2}\xi_\Sigma|_2\right]\\
&\le \frac{C\sqrt p}{\sqrt{1-e^{-2t}}}
\|\xi\|_{3p,r}\|\xi_\Sigma\|_{3p,2}.
\end{aligned}
\]
Here the expectation in the first line integrates a copy of $\xi$
independent of $Z$. Conditional Jensen's inequality therefore yields
\[
\|\rho_t(F_t)\|_{p,r}
\le e^{-t}\|\xi\|_{3p,r}
+\frac{C\sqrt p\,e^{-2t}}{\sqrt{1-e^{-2t}}}
\|\xi\|_{3p,r}\|\xi_\Sigma\|_{3p,2}.
\]
Using
\[
\int_0^\infty e^{-t}\,dt=1,
\qquad
\int_0^\infty\frac{e^{-2t}}{\sqrt{1-e^{-2t}}}\,dt=1,
\]
we obtain from \cref{eq:masterbound}, $p\ge1$, and
$\|\xi_\Sigma\|_{3p,2}^2>1/2$ that
\[
\begin{aligned}
\mcl W_{p,r}(\mcl L(W),N(0,\Sigma))
&\le C\|\xi\|_{3p,r}
\left(1+\sqrt p\,\|\xi_\Sigma\|_{3p,2}\right)\\
&\le Cp\|\xi\|_{3p,r}\|\xi_\Sigma\|_{3p,2}^2.
\end{aligned}
\]

For the remainder of the proof, we assume without loss of generality that  $\|\xi_\Sigma\|_{3p,2}^2\le1/2$ and $\Sigma-V\succ0$ (the general case that $\Sigma-V\succeq 0$ follows by a simple approximation argument).  
Then
\ben{\label{eq:ctlower}
\|V_{1,\Sigma}\|_2
\le \mathbb E|\xi_\Sigma|_2^2
\le \|\xi_\Sigma\|_{3p,2}^2\le\frac12,
\qquad c_t\ge\frac12.
}
To retain the third-order bound, define
\[
\widetilde\xi_{\Sigma,t}
:=\xi_\Sigma\mathbf 1_{\{|\xi_\Sigma|_2\le \eta_t(p)\}},
\qquad
\eta_t(p):=\sqrt{\frac{e^{2t}-1}{p}},
\]
and
\begin{equation}\label{eq:taus}
\tau_s
:=
\xi+
\xi\sum_{k=1}^{\infty}
\frac{e^{-ks}}{k!}
\left\langle
    \left(\widetilde I_{d,s}^{-1/2}\widetilde\xi_{\Sigma,s}\right)^{\otimes k},
    \frac{\nabla^k\phi(Z)}{\phi(Z)}
\right\rangle,
\end{equation}
where, for a vector $v\in\mathbb{R}^d$, $v^{\otimes k}$ denotes the $k$-fold tensor product $v\otimes\cdots\otimes v$, 
and $\nabla^k\phi(z)$ denotes the $k$-th order derivative tensor of $\phi$ at $z$. 
The pairing $\langle \cdot,\cdot\rangle$ represents the natural contraction between tensors of the same order. 
For the fixed \(t\), write
\(\widetilde\xi_\Sigma:=\widetilde\xi_{\Sigma,t}\) for convenience of notation.
By \cref{lem:1},
\besn{\label{eq:bonis}
&\mathbb E\!\left[
   \xi-e^{-t}V_1\Sigma^{-1/2}\widetilde I_{d,t}^{-1/2}Z
   \,\middle|\,
   e^{-t}\widetilde I_{d,t}^{-1/2}\widetilde\xi_\Sigma+Z
\right] \\
&=
\mathbb E\!\left[
   \xi-e^{-t}V_1\Sigma^{-1/2}\widetilde I_{d,t}^{-1/2}Z-\tau_t
   \,\middle|\,
   e^{-t}\widetilde I_{d,t}^{-1/2}\widetilde\xi_\Sigma+Z
\right] \\
&=
\mathbb E\!\left[
   -e^{-t}V_1\Sigma^{-1/2}\widetilde I_{d,t}^{-1/2}Z
   -\xi\sum_{k=1}^{\infty}\frac{e^{-kt}}{k!}
   \left\langle
      \left(\widetilde I_{d,t}^{-1/2}\widetilde\xi_\Sigma\right)^{\otimes k},
      \frac{\nabla^k\phi(Z)}{\phi(Z)}
   \right\rangle
   \,\middle|\,
   e^{-t}\widetilde I_{d,t}^{-1/2}\widetilde\xi_\Sigma+Z
\right] \\
&=
\mathbb E\!\left[
   e^{-t}
   \left(
      \xi\widetilde\xi_\Sigma^\top
      -\mathbb E[\xi\widetilde\xi_\Sigma^\top]
   \right)
   \widetilde I_{d,t}^{-1/2}Z
   -e^{-t}
   \mathbb E\!\left[
      \xi\xi_\Sigma^\top\mathbf 1_{\{|\xi_\Sigma|_2>\eta_t(p)\}}
   \right]
   \widetilde I_{d,t}^{-1/2}Z
\right. \\
&\hspace{4.5cm}\left.
   -\xi
   \sum_{k=2}^{\infty}\frac{e^{-kt}}{k!}
   \left\langle
      \left(\widetilde I_{d,t}^{-1/2}\widetilde\xi_\Sigma\right)^{\otimes k},
      \frac{\nabla^k\phi(Z)}{\phi(Z)}
   \right\rangle
   \,\middle|\,
   e^{-t}\widetilde I_{d,t}^{-1/2}\widetilde\xi_\Sigma+Z
\right].
}
The regression identity and \cref{eq:bonis} give
\[
\|\rho_t(F_t)\|_{p,r}
\le e^{-t}(\mathcal R+\mathcal R_1+\mathcal R_2+\mathcal R_3),
\]
where
\[
\begin{aligned}
\mathcal R
&:=
\Bigg\|
\mathbb E\!\left[
   \xi-e^{-t}V_1\Sigma^{-1/2}\widetilde I_{d,t}^{-1/2}Z
   \,\middle|\,
   Z+e^{-t}\widetilde I_{d,t}^{-1/2}\xi_\Sigma
\right] \\
&\hspace{2.6cm}
-
\mathbb E\!\left[
   \xi-e^{-t}V_1\Sigma^{-1/2}\widetilde I_{d,t}^{-1/2}Z
   \,\middle|\,
   Z+e^{-t}\widetilde I_{d,t}^{-1/2}\widetilde\xi_\Sigma
\right]
\Bigg\|_{p,r}, \\
\mathcal R_1
&:= e^{-t}
\left\|
\mathbb E\!\left[
   \left(
      \xi\widetilde\xi_\Sigma^\top
      -\mathbb E[\xi\widetilde\xi_\Sigma^\top]
   \right)
   \widetilde I_{d,t}^{-1/2}Z
   \,\middle|\,
   e^{-t}\widetilde I_{d,t}^{-1/2}\widetilde\xi_\Sigma+Z
\right]
\right\|_{p,r}, \\
\mathcal R_2
&:= e^{-t}
\left\|
\mathbb E\!\left[
   \mathbb E\!\left[
      \xi\xi_\Sigma^\top\mathbf 1_{\{|\xi_\Sigma|_2>\eta_t(p)\}}
   \right]
   \widetilde I_{d,t}^{-1/2}Z
   \,\middle|\,
   e^{-t}\widetilde I_{d,t}^{-1/2}\widetilde\xi_\Sigma+Z
\right]
\right\|_{p,r}, \\
\mathcal R_3
&:=
\left\|
   \xi
   \sum_{k=2}^{\infty}
   \frac{e^{-kt}}{k!}
   \left\langle
      \left(\widetilde I_{d,t}^{-1/2}\widetilde\xi_\Sigma\right)^{\otimes k},
      \frac{\nabla^k\phi(Z)}{\phi(Z)}
   \right\rangle
\right\|_{p,r}.
\end{aligned}
\]

\noindent\textbf{Step 1: Control of \(\mathcal R\).}

By the triangle inequality,
\[
\mathcal R\le \mathcal R_\xi+\mathcal R_Z,
\]
where
\[
\begin{aligned}
\mathcal R_\xi
&:= \Bigg\|
\mathbb E\!\left[
    \xi
    \,\middle|\,
    Z+e^{-t}\widetilde I_{d,t}^{-1/2}\xi_\Sigma
\right]
-
\mathbb E\!\left[
    \xi
    \,\middle|\,
    Z+e^{-t}\widetilde I_{d,t}^{-1/2}\widetilde\xi_\Sigma
\right]
\Bigg\|_{p,r}, \\
\mathcal R_Z
&:= e^{-t}\Bigg\|
V_1\Sigma^{-1/2}\widetilde I_{d,t}^{-1/2}
\left\{
\mathbb E\!\left[
    Z
    \,\middle|\,
    Z+e^{-t}\widetilde I_{d,t}^{-1/2}\xi_\Sigma
\right]
-
\mathbb E\!\left[
    Z
    \,\middle|\,
    Z+e^{-t}\widetilde I_{d,t}^{-1/2}\widetilde\xi_\Sigma
\right]
\right\}
\Bigg\|_{p,r}.
\end{aligned}
\]

We first bound \(\mathcal R_\xi\). We will apply \cref{lem:2} with
\[
\begin{gathered}
X=e^{-t}\widetilde I_{d,t}^{-1/2}\widetilde\xi_\Sigma,
\qquad
Y=e^{-t}\widetilde I_{d,t}^{-1/2}(\xi_\Sigma-\widetilde\xi_\Sigma),\\[-1pt]
H(x,y,z)=e^t\Sigma^{1/2}\widetilde I_{d,t}^{1/2}(x+y),\\
\alpha_0=3p,
\quad \beta_0=\frac{3p}{2},
\quad \alpha_i=\beta_i=\gamma_i=3p\ (i=1,2).
\end{gathered}
\]
Then
\[
\begin{gathered}
Z+X+Y=Z+e^{-t}\widetilde I_{d,t}^{-1/2}\xi_\Sigma,
\qquad
Z+X=Z+e^{-t}\widetilde I_{d,t}^{-1/2}\widetilde\xi_\Sigma,\\[-1pt]
H(X,Y,Z)=\Sigma^{1/2}\xi_\Sigma=\xi,
\quad \partial_zH=0.
\end{gathered}
\]
Lemma~\ref{lem:2}, with $g\sim N(0,1)$ and $\|g\|_{3p/2}\le C\sqrt p$, gives
\besn{
\mathcal R_\xi
&\le C\|\xi\|_{3p,r}\Bigg[
\sqrt p\,e^{-t}
\left\|\widetilde I_{d,t}^{-1/2}
(\xi_\Sigma-\widetilde\xi_\Sigma)\right\|_{\frac{3p}{2},2} \\
&\qquad+e^{-2t}
\left\|\widetilde I_{d,t}^{-1/2}\widetilde\xi_\Sigma\right\|_{3p,2}
\left\|\widetilde I_{d,t}^{-1/2}
(\xi_\Sigma-\widetilde\xi_\Sigma)\right\|_{3p,2} \\
&\qquad+e^{-2t}
\left\|\widetilde I_{d,t}^{-1/2}
(\xi_\Sigma-\widetilde\xi_\Sigma)\right\|_{3p,2}^{2}\Bigg].
}
Since $\xi_\Sigma-\widetilde\xi_\Sigma
=\xi_\Sigma\mathbf 1_{\{|\xi_\Sigma|_2>\eta_t(p)\}}$, \cref{eq:Idbound} yields
\besn{
\left\|\widetilde I_{d,t}^{-1/2}
(\xi_\Sigma-\widetilde\xi_\Sigma)\right\|_{\frac{3p}{2},2}
&\le c_t^{-1/2}\eta_t(p)^{-1}
\left\||\xi_\Sigma|_2^2\right\|_{\frac{3p}{2}}
=\sqrt{\frac{p}{c_t(e^{2t}-1)}}\,\|\xi_\Sigma\|_{3p,2}^{2},\\
\left\|\widetilde I_{d,t}^{-1/2}\widetilde\xi_\Sigma\right\|_{3p,2}
&\le c_t^{-1/2}\|\xi_\Sigma\|_{3p,2},\\
\left\|\widetilde I_{d,t}^{-1/2}
(\xi_\Sigma-\widetilde\xi_\Sigma)\right\|_{3p,2}
&\le c_t^{-1/2}\|\xi_\Sigma\|_{3p,2}.
}
Thus
\[
\mathcal R_\xi
\le Cp\left\{\frac{e^{-t}}{\sqrt{c_t(e^{2t}-1)}}
+\frac{e^{-2t}}{c_t}\right\}
\|\xi\|_{3p,r}\|\xi_\Sigma\|_{3p,2}^{2}.
\]

It remains to bound \(\mathcal R_Z\). 
Applying the identity \(\mathbb E[Z\mid Z+U]=Z+U-\mathbb E[U\mid Z+U]\) with
\(U=e^{-t}\widetilde I_{d,t}^{-1/2}\xi_\Sigma\) and
\(U=e^{-t}\widetilde I_{d,t}^{-1/2}\widetilde\xi_\Sigma\) gives
\[
\begin{aligned}
&\mathbb E\!\left[
Z
\,\middle|\,
Z+e^{-t}\widetilde I_{d,t}^{-1/2}\xi_\Sigma
\right]
-
\mathbb E\!\left[
Z
\,\middle|\,
Z+e^{-t}\widetilde I_{d,t}^{-1/2}\widetilde\xi_\Sigma
\right] \\
&=
e^{-t}\widetilde I_{d,t}^{-1/2}(\xi_\Sigma-\widetilde\xi_\Sigma) \\
&\quad
-e^{-t}\widetilde I_{d,t}^{-1/2}
\left\{
\mathbb E\!\left[
\xi_\Sigma
\,\middle|\,
Z+e^{-t}\widetilde I_{d,t}^{-1/2}\xi_\Sigma
\right]
-
\mathbb E\!\left[
\widetilde\xi_\Sigma
\,\middle|\,
Z+e^{-t}\widetilde I_{d,t}^{-1/2}\widetilde\xi_\Sigma
\right]
\right\}.
\end{aligned}
\]
For every deterministic $v$, the rank-one identity gives
\[
\begin{aligned}
V_1\Sigma^{-1/2}\widetilde I_{d,t}^{-1}v
&=\mathbb E\!\left[\xi
\left\langle\widetilde I_{d,t}^{-1}\xi_\Sigma,v\right\rangle\right],\\
\left|V_1\Sigma^{-1/2}\widetilde I_{d,t}^{-1}v\right|_r
&\le \mathbb E\!\left[|\xi|_r
|\widetilde I_{d,t}^{-1}\xi_\Sigma|_2\right]|v|_2.
\end{aligned}
\]
Conditional Jensen's and H\"older's inequalities therefore imply
\[
\begin{aligned}
\mathcal R_Z
&\le e^{-2t}\mathbb E\!\left[|\xi|_r
|\widetilde I_{d,t}^{-1}\xi_\Sigma|_2\right]
\left(\|\xi_\Sigma-\widetilde\xi_\Sigma\|_{p,2}
+\|\xi_\Sigma\|_{p,2}+\|\widetilde\xi_\Sigma\|_{p,2}\right)\\
&\le\frac{3e^{-2t}}{c_t}
\|\xi\|_{3p,r}\|\xi_\Sigma\|_{3p,2}^2.
\end{aligned}
\]
Combining the bounds for $\mathcal R_\xi$ and $\mathcal R_Z$, using $c_t\le1$ and $p\ge1$, gives
\[
\mathcal R
\le Cp\left\{\frac{e^{-t}}{\sqrt{e^{2t}-1}\,c_t}
+\frac{e^{-2t}}{c_t}\right\}
\|\xi\|_{3p,r}\|\xi_\Sigma\|_{3p,2}^2.
\]

\noindent\textbf{Step 2: Control of \(\mathcal R_1\).}\par\nopagebreak

Let \(\tilde V_1:=\mathbb E(\xi\tilde\xi_\Sigma^\top)\). Then
\[
\mathbb E\!\left[\left(\xi\tilde\xi_\Sigma^\top-\tilde V_1\right)
\tilde I_{d,t}^{-1/2}Z\mid Z\right]=0.
\]
We will apply \cref{lem:2}, using \(\xi\) as the auxiliary random element, with
\[
\begin{gathered}
X=0,
\qquad Y=e^{-t}\tilde I_{d,t}^{-1/2}\tilde\xi_\Sigma,
\qquad
H(x,y,z)=\left(\xi\tilde\xi_\Sigma^\top-\tilde V_1\right)\tilde I_{d,t}^{-1/2}z,\\[-1pt]
Z+X+Y=Z+e^{-t}\tilde I_{d,t}^{-1/2}\tilde\xi_\Sigma,
\quad Z+X=Z,
\quad
\partial_zH=\left(\xi\tilde\xi_\Sigma^\top-\tilde V_1\right)\tilde I_{d,t}^{-1/2},\\[-1pt]
\rho=2,
\quad a=\frac{3p}{2},
\quad c=3p,
\quad \alpha_0=\frac{3p}{2},
\quad \beta_0=3p,
\quad \alpha_2=\beta_2=\gamma_2=3p.
\end{gathered}
\]
Lemma~\ref{lem:2} gives
\bes{
\frac{\mathcal R_1}{e^{-t}}
&\le
2e^{-t}
\left\|
\left(\xi\tilde\xi_\Sigma^\top-\tilde V_1\right)
\tilde I_{d,t}^{-1/2}
\right\|_{\frac{3p}{2},\,2\to r}
\left\|
\tilde I_{d,t}^{-1/2}\tilde\xi_\Sigma
\right\|_{3p,2} \\
&\quad+
4e^{-t}
\left\|
\left(\xi\tilde\xi_\Sigma^\top-\tilde V_1\right)
\tilde I_{d,t}^{-1/2}Z
\right\|_{\frac{3p}{2},r}
\left\|
\tilde I_{d,t}^{-1/2}\tilde\xi_\Sigma
\right\|_{3p,2}
\|g\|_{3p} \\
&\quad+
2e^{-2t}
\left\|
\left(\xi\tilde\xi_\Sigma^\top-\tilde V_1\right)
\tilde I_{d,t}^{-1/2}Z
\right\|_{3p,r}
\left\|
\tilde I_{d,t}^{-1/2}\tilde\xi_\Sigma
\right\|_{3p,2}^{2}.
}

We estimate the three factors on the right-hand side. First, by
\(\|uv^\top\|_{2\to r}\le |u|_r|v|_2\) and Jensen's inequality,
\[
\begin{aligned}
&\left\|
\left(\xi\tilde\xi_\Sigma^\top-\tilde V_1\right)
\tilde I_{d,t}^{-1/2}
\right\|_{\frac{3p}{2},\,2\to r} \\
&\quad\le
\left\|
\xi\tilde\xi_\Sigma^\top\tilde I_{d,t}^{-1/2}
\right\|_{\frac{3p}{2},\,2\to r}
+
\left\|
\tilde V_1\tilde I_{d,t}^{-1/2}
\right\|_{2\to r} \\
&\quad\le
\|\xi\|_{3p,r}
\left\|
\tilde I_{d,t}^{-1/2}\tilde\xi_\Sigma
\right\|_{3p,2}
+
\mathbb{E}\!\left[
|\xi|_r
\left|
\tilde I_{d,t}^{-1/2}\tilde\xi_\Sigma
\right|_2
\right] \\
&\quad\le
2\|\xi\|_{3p,r}
\left\|
\tilde I_{d,t}^{-1/2}\tilde\xi_\Sigma
\right\|_{3p,2}.
\end{aligned}
\]
Second, since \(Z\) is independent of \((\xi,\tilde\xi_\Sigma)\),
\[
\begin{aligned}
&\left\|
\left(\xi\tilde\xi_\Sigma^\top-\tilde V_1\right)
\tilde I_{d,t}^{-1/2}Z
\right\|_{\frac{3p}{2},r} \\
&\quad\le
C
\left\|
|\xi|_r
\left|
\left\langle
\tilde I_{d,t}^{-1/2}\tilde\xi_\Sigma,Z
\right\rangle
\right|
\right\|_{\frac{3p}{2}} \\
&\quad\le
C\|g\|_{3p}
\|\xi\|_{3p,r}
\left\|
\tilde I_{d,t}^{-1/2}\tilde\xi_\Sigma
\right\|_{3p,2} \\
&\quad\le
C\sqrt p\,
\|\xi\|_{3p,r}
\left\|
\tilde I_{d,t}^{-1/2}\tilde\xi_\Sigma
\right\|_{3p,2}.
\end{aligned}
\]
Finally, using \(|\tilde\xi_\Sigma|_2\le \eta_t(p)\), we have
\besn{
&\left\|
\left(\xi\tilde\xi_\Sigma^\top-\tilde V_1\right)
\tilde I_{d,t}^{-1/2}Z
\right\|_{3p,r} \\
&\quad\le
C
\left\|
|\xi|_r
\left|
\left\langle
\tilde I_{d,t}^{-1/2}\tilde\xi_\Sigma,Z
\right\rangle
\right|
\right\|_{3p} \\
&\quad\le
C\|g\|_{3p}
\left\|
\left|
\tilde I_{d,t}^{-1/2}\tilde\xi_\Sigma
\right|_2
|\xi|_r
\right\|_{3p} \\
&\quad\le
C\sqrt p\,
\|\tilde I_{d,t}^{-1/2}\|_2
\eta_t(p)
\|\xi\|_{3p,r}.
}
Combining the three estimates and using \cref{eq:Idbound}, we obtain
\[
\begin{aligned}
\mathcal R_1
&\le Cp e^{-2t}\left(1+
e^{-t}\|\tilde I_{d,t}^{-1/2}\|_2\eta_t(p)\right)
\|\xi\|_{3p,r}
\left\|\tilde I_{d,t}^{-1/2}\tilde\xi_\Sigma\right\|_{3p,2}^{2}\\
&\le Cp\left\{
\frac{e^{-2t}}{c_t}
+\frac{e^{-3t}\eta_t(p)}{c_t^{3/2}}\right\}
\|\xi\|_{3p,r}\|\xi_\Sigma\|_{3p,2}^{2}.
\end{aligned}
\]

\noindent\textbf{Step 3: Control of \(\mathcal R_2\).}

Conditional Jensen's, Markov's and H\"older's inequalities, together with
$\|g\|_p\le\sqrt p$ for $g\sim N(0,1)$, give
\[
\begin{aligned}
\mathcal R_2
&\le e^{-t}\left\|
\xi\xi_\Sigma^\top\mathbf 1_{\{|\xi_\Sigma|_2>\eta_t(p)\}}
\tilde I_{d,t}^{-1/2}Z\right\|_{p,r}\\
&\le \frac{e^{-t}\sqrt p}{\sqrt{e^{2t}-1}}
\left\|\tilde I_{d,t}^{-1/2}\right\|_2
\left\||\xi|_r|\xi_\Sigma|_2^2\right\|_p\|g\|_p\\
&\le\frac{Cp e^{-t}}{\sqrt{c_t(e^{2t}-1)}}
\|\xi\|_{3p,r}\|\xi_\Sigma\|_{3p,2}^{2}.
\end{aligned}
\]


\noindent\textbf{Step 4: Control of \(\mathcal R_3\).}  

Here we use the following additional conventions. 
For a multi-index $\alpha=(\alpha_1,\dots,\alpha_d)\in\mathbb{N}^d$ of nonnegative integers, we write $|\alpha|=\sum_{j=1}^d \alpha_j$, $\alpha!=\alpha_1!\cdots\alpha_d!$ and $x^\alpha = x_1^{\alpha_1}\cdots x_d^{\alpha_d}$; 
$H_\alpha(Z)$ denotes the Hermite polynomial associated with $\alpha$. We have
\[
\begin{aligned}
\mathcal R_3
&= \left\lVert \xi\,
   \sum_{k=2}^{\infty} \frac{e^{-k t}}{k!}
   \left\langle\bigl(\tilde I_{d,t}^{-1/2}\tilde{\xi}_\Sigma\bigr)^{\otimes k}, \frac{\nabla^k \phi(Z)}{\phi(Z)}\right\rangle
   \right\rVert_{p,r} \\[6pt]
&\leqslant \sum_{k=2}^{\infty} \frac{e^{-k t}}{k!}
   \left\lVert \xi\,
   \left\langle\bigl(\tilde{I}_{d,t}^{-1/2}\tilde{\xi}_\Sigma\bigr)^{\otimes k}, \frac{\nabla^k \phi(Z)}{\phi(Z)}\right\rangle
   \right\rVert_{p,r} \\[6pt]
&=\sum_{k=2}^{\infty} \frac{e^{-k t}}{k!}
   \left\lVert \xi\sum_{|\alpha|=k}
   \frac{k!}{\alpha!}\,\bigl(\tilde{I}_{d,t}^{-1/2}\tilde{\xi}_\Sigma\bigr)^{\alpha}\,H_\alpha(Z)
   \right\rVert_{p,r} \\[6pt]
&=\sum_{k=2}^{\infty} \frac{e^{-k t}}{k!}
   \left\lVert \Biggl(\sum_{|\alpha|=k}
   \frac{k!}{\alpha!}\,\bigl(\tilde I_{d,t}^{-1/2}\tilde{\xi}_\Sigma\bigr)^{\alpha}\,H_\alpha(Z)\Biggr)
   \,|\xi|_r \right\rVert_p \\[6pt]
&=\sum_{k=2}^{\infty} \frac{e^{-k t}}{k!}
   \Biggl(\mathbb{E}\Biggl[
   \mathbb{E}\Biggl(\Bigl\lvert\sum_{|\alpha|=k} \frac{k!}{\alpha!}
   \bigl(\tilde{I}_{d,t}^{-1/2}\tilde{\xi}_\Sigma\bigr)^\alpha H_\alpha(Z)\Bigr\rvert^p
   \,|\xi|_r^p \,\Big|\, \xi\Biggr)\Biggr]\Biggr)^{1/p}.
\end{aligned}
\]
From \cite[Lemma~3]{bonis2020stein}, we have, for any \( p \geq 1 \),
\ben{\label{lem:3}
\Bigl( \mathbb{E} \bigl| \sum_{\alpha} M_\alpha H_\alpha(Z) \bigr|^p  \Bigr)^{2/p} 
\leq \sum_{\alpha} \max\{1,\,p-1\}^{|\alpha|} \, \alpha! \, |M_\alpha|^2 ,
}
where 
\(M_\alpha\in\mathbb R\), \(\alpha\in\mathbb N^d\), have only finitely many
nonzero coefficients. 

Conditioning on \(\xi\), using the independence of \(Z\), and applying
\cref{lem:3},
\[
\begin{aligned}
\mathcal R_3
& \leqslant 
\sum_{k=2}^{\infty} \frac{e^{-k t}}{k!}
\Biggl(\mathbb{E}\Biggl[
\sum_{|\alpha|=k} \max(1,p-1)^k \alpha!\,
|\xi|_r^2\,
\Bigl\lvert \frac{k!}{\alpha!}\,(\tilde I_{d,t}^{-1/2}\tilde{\xi}_\Sigma)^\alpha \Bigr\rvert^2
\Biggr]^{p/2}\Biggr)^{1/p} \\
&=\sum_{k=2}^{\infty} \frac{e^{-k t}}{\sqrt{k!}}\,
   \max(1,p-1)^{k/2}
   \Biggl(\mathbb{E}\Biggl[\sum_{|\alpha|=k} \frac{k!}{\alpha!}\,
   | \xi|_r^2 \,
   \bigl\lvert (\tilde{I}_{d,t}^{-1/2}\tilde{\xi}_\Sigma)^{\alpha} \bigr\rvert^2
   \Biggr]^{p/2}\Biggr)^{1/p} \\[6pt]
&=\sum_{k=2}^{\infty} \frac{e^{-k t}}{\sqrt{k!}}\,
   \max(1,p-1)^{k/2}\,
   \bigg\lVert \,| \xi|_r\,| \tilde{I}_{d,t}^{-1/2}\tilde{\xi}_\Sigma|_2^k \,\bigg\rVert_p\\[6pt]
&\leqslant \sum_{k=2}^{\infty} \frac{e^{-k t}}{\sqrt{k!}}\,
   \max(1,p-1)^{k/2}\,
   \eta_t(p)^{\,k-2}\,
   \lVert \tilde{I}_{d,t}^{-1/2}\rVert_2^k\,
   \bigg\lVert \,| \xi|_r\,| \xi_\Sigma|_2^2 \,\bigg\rVert_p.
\end{aligned}
\]

Since $\max(1,p-1)\le p$, the preceding bound implies
\[
\begin{aligned}
\mathcal R_3
&\le \frac{pe^{-2t}}{c_t}
\sum_{k=2}^{\infty}\frac1{\sqrt{k!}}
\left(\frac{1-e^{-2t}}{c_t}\right)^{(k-2)/2}
\left\||\xi|_r|\xi_\Sigma|_2^2\right\|_p\\
&\le \frac{Cpe^{-2t}}{c_t}
\|\xi\|_{3p,r}\|\xi_\Sigma\|_{3p,2}^2,
\end{aligned}
\]
because $(1-e^{-2t})/c_t\le1$ and
$\sum_{k=2}^{\infty}1/\sqrt{k!}<\infty$.

Recall the decomposition
\[
\|\rho_t(F_t)\|_{p,r}
\le
e^{-t}(\mathcal R+\mathcal R_1+\mathcal R_2+\mathcal R_3).
\]
The elementary estimates
\[
\int_0^\infty
\left\{
\frac{e^{-2t}}{\sqrt{e^{2t}-1}\,c_t}
+
\frac{e^{-3t}}{c_t}
+
\frac{e^{-4t}\eta_t(p)}{c_t^{3/2}}
\right\}\,dt
\le
C
\]
(where we used \(c_t\geq 1/2\) from \cref{eq:ctlower}) and the preceding bounds yield
\[
\int_0^\infty \|\rho_t(F_t)\|_{p,r}\,dt
\le
Cp
\|\xi\|_{3p,r}\|\xi_\Sigma\|_{3p,2}^2.
\]
The claim follows from \cref{eq:masterbound}.
\end{proof}

\begin{proof}[Proof of \cref{thm:3}]
Let \(Z_0,Z_1,\ldots,Z_n\) be i.i.d. standard normal random vectors,
independent of \(\mathcal F_n\), and put \(Y=M^{1/2}Z_0\).  Define
\[
        G:=\sum_{i=1}^n V_i^{1/2}Z_i,
        \qquad
        T_k:=\sum_{i=k}^n V_i^{1/2}Z_i,
        \qquad 1\le k\le n+1,
\]
where \(T_{n+1}=0\). Since \(\sum_{i=1}^nV_i=\Sigma=\bar\Sigma\) a.s.,
conditionally on \(\mathcal F_n\) the random vector \(G\) is Gaussian with
covariance \(\Sigma\). Hence \(G\sim N(0,\Sigma)\).

The identity
\[
        \Pi_k=\Sigma-\sum_{i=1}^{k-1}V_i
\]
shows that \(\Pi_k\) and \(\Pi_{k+1}\) are \(\mathcal F_{k-1}\)-measurable.
Moreover, conditionally on \(\mathcal F_{k-1}\),
\[
        T_{k+1}+Y\sim N(0,\Pi_{k+1}+M),
        \qquad
        V_k^{1/2}Z_k+T_{k+1}+Y\sim N(0,\Pi_k+M).
\]
Also \(T_{k+1}+Y\) is conditionally independent of \(\xi_k\) given
\(\mathcal F_{k-1}\). 

By the triangle inequality, we have
\besn{\label{eq:triangle}
&\mcl W_{p,r}\bigl(\mathcal L(S_n),N(0,\Sigma)\bigr)\\
&\le W_{p,r}\bigl(\mathcal L(S_n+Y),\mathcal L(G+Y))+W_{p,r}\bigl(\mathcal L(S_n+Y),\mathcal L(S_n))+W_{p,r}\bigl(\mathcal L(G),\mathcal L(G+Y))\\
&\le
\mcl W_{p,r}\bigl(\mathcal L(S_n+Y),\mathcal L(G+Y)\bigr)
     +2\|Y\|_{p,r}  \\
&=
\mcl W_{p,r}\bigl(\mathcal L(S_n+Y),\mathcal L(G+Y)\bigr)
     +2\|M^{1/2}Z\|_{p,r},
}
where \(Z\sim N(0,I_d)\).

Let
\[
        U_k:=S_k+T_{k+1}+Y,
        \qquad 0\le k\le n .
\]
Then \(U_n=S_n+Y\) and \(U_0=G+Y\). Hence Lindeberg's telescoping argument
and the triangle inequality give
\[
\mcl W_{p,r}\bigl(\mathcal L(S_n+Y),\mathcal L(G+Y)\bigr)
\le
\sum_{k=1}^n R_k,
\]
where
\[
        R_k:=\mcl W_{p,r}\bigl(\mathcal L(U_k),\mathcal L(U_{k-1})\bigr).
\]
Integrating conditional optimal couplings given \(\mathcal F_{k-1}\) yields
\[
\begin{aligned}
R_k
&\le
\left\{
\mathbb E\left[
\mcl W_{p,r}^p\!\left(
\mathcal L(U_k\mid\mathcal F_{k-1}),
\mathcal L(U_{k-1}\mid\mathcal F_{k-1})
\right)
\right]
\right\}^{1/p}   \\
&=
\left\{
\mathbb E\left[
\mcl W_{p,r}^p\!\left(
\mathcal L(\xi_k+T_{k+1}+Y\mid\mathcal F_{k-1}),
\mathcal L(V_k^{1/2}Z_k+T_{k+1}+Y\mid\mathcal F_{k-1})
\right)
\right]
\right\}^{1/p}.
\end{aligned}
\]

We now apply \cref{prop:1} conditionally on \(\mathcal F_{k-1}\), with
\[
        \xi=\xi_k,\qquad V=V_k,\qquad \Sigma=\Pi_k+M.
\]
Here the total covariance is \(\Pi_k+M\succ0\), while the independent
Gaussian covariance is \(\Pi_{k+1}+M\succeq0\), including zero when
\(k=n\) and \(M=0\). The universal constant in \cref{prop:1} is uniform
over the conditional distributions.
Therefore, almost surely,
\[
\begin{aligned}
&\mcl W_{p,r}\!\left(
\mathcal L(\xi_k+T_{k+1}+Y\mid\mathcal F_{k-1}),
\mathcal L(V_k^{1/2}Z_k+T_{k+1}+Y\mid\mathcal F_{k-1})
\right)  \\
&\qquad\le
Cp
\|\xi_k\mid\mathcal F_{k-1}\|_{3p,r}
\|(\Pi_k+M)^{-1/2}\xi_k\mid\mathcal F_{k-1}\|_{3p,2}^{2}.
\end{aligned}
\]
Substituting this estimate into the preceding bound for \(R_k\), summing over
\(k\), and using the smoothing inequality \cref{eq:triangle} proves \eqref{eq:thm3}.
\end{proof}

\begin{proof}[Proof of \cref{cor:2}]
Since $\Pi_k=\bar\Sigma-\sum_{i<k}V_i$ is
$\mathcal F_{k-1}$-measurable, the covariance lower bounds give
\[
\begin{gathered}
\Pi_k\succeq\frac{\alpha(n-k+1)}n\bar\Sigma\succ0,
\qquad\Pi_k^{-1}\preceq\frac{n}{\alpha(n-k+1)}\bar\Sigma^{-1},\\
\|\Pi_k^{-1/2}\xi_k\mid\mathcal F_{k-1}\|_{3p,2}^2
\le\frac{n}{\alpha(n-k+1)}
\|\bar\Sigma^{-1/2}\xi_k\mid\mathcal F_{k-1}\|_{3p,2}^2.
\end{gathered}
\]
Applying \cref{thm:3} with $M=0$ proves the bound in \cref{eq:cor2.1-1}.
For \cref{eq:cor2.1-2}, H\"older's inequality in the outer expectation and the condition \cref{eq:cor2.1cond} give
\[
\begin{aligned}
&\left\|
   \|\xi_k\mid\mathcal F_{k-1}\|_{3p,r}
   \|\bar\Sigma^{-1/2}\xi_k\mid\mathcal F_{k-1}\|_{3p,2}^2
\right\|_p\\
&\quad\le
\left\|\|\xi_k\mid\mathcal F_{k-1}\|_{3p,r}\right\|_{3p}
\left\|\|\bar\Sigma^{-1/2}\xi_k\mid\mathcal F_{k-1}\|_{3p,2}\right\|_{3p}^2\\
&\quad=
\|\xi_k\|_{3p,r}\|\bar\Sigma^{-1/2}\xi_k\|_{3p,2}^2
\le\frac{K}{n^{3/2}}.
\end{aligned}
\]
Substitution into \cref{eq:cor2.1-1} yields
\[
\mcl W_{p,r}\bigl(\mathcal L(S_n),N(0,\bar\Sigma)\bigr)
\le\frac{CpK}{\alpha\sqrt n}\sum_{k=1}^n\frac1{n-k+1}
\le\frac{CpK}{\alpha}\frac{\log(en)}{\sqrt n}.
\]
\end{proof}

\begin{proof}[Proof of \cref{cor:1}]
By \cref{thm:3} with $M=\sigma^2I_d$, for every $\sigma>0$,
\begin{equation}\label{eq:cor22-sigma-explicit}
\begin{gathered}
\Pi_k+\sigma^2I_d\succeq\sigma^2I_d\succ0
\quad(1\le k\le n),\\
\mcl W_{p,r}\!\left(\mathcal L(S_n),N(0,\bar\Sigma)\right)
\le\frac{CpA_{p,r}}{\sigma^2}+2\sigma\|Z\|_{p,r}.
\end{gathered}
\end{equation}
If $A_{p,r}>0$, choosing
\[
\sigma=\left(\frac{pA_{p,r}}{\|Z\|_{p,r}}\right)^{1/3}
\]
gives \eqref{eq:cor1}. If $A_{p,r}=0$, let $\sigma\downarrow0$ in
\eqref{eq:cor22-sigma-explicit}.
\end{proof}


\begin{proof}[Proof of \cref{thm:2}]
Let $Z_0\sim N(0,I_d)$ be independent of $\mathcal F_n$ and append
$H^{1/2}Z_0$ as the $(n+1)$-st increment, enlarging the filtration by
$Z_0$ only at time $n+1$. Then
\[
\begin{gathered}
\widetilde S:=S_n+H^{1/2}Z_0,
\qquad\Sigma+H=\bar\Sigma+\Delta_0,\\
\Pi_k+H=\bar\Sigma+\Delta_0-\sum_{i<k}V_i,
\qquad1\le k\le n.
\end{gathered}
\]
Thus $\Pi_k+H$ is $\mathcal F_{k-1}$-measurable. At step $n+1$, both smoothed conditional laws are $N(0,H+M)$, this step introduces zero error in the Lindeberg's swapping argument.
For $k\le n$, inverse order and quadratic forms give
\[
\begin{aligned}
\Pi_k+H+M&\succeq\Pi_k+M\succ0,\\
(\Pi_k+H+M)^{-1}&\preceq(\Pi_k+M)^{-1},\\
\xi_k^\top(\Pi_k+H+M)^{-1}\xi_k
&\le\xi_k^\top(\Pi_k+M)^{-1}\xi_k,\\
\|(\Pi_k+H+M)^{-1/2}\xi_k\mid\mathcal F_{k-1}\|_{3p,2}^2
&\le\|(\Pi_k+M)^{-1/2}\xi_k\mid\mathcal F_{k-1}\|_{3p,2}^2.
\end{aligned}
\]
Thus, from \cref{thm:3},
\[
\begin{aligned}
\mcl W_{p,r}\bigl(\mathcal L(\widetilde S),N(0,\bar\Sigma+\Delta_0)\bigr)
\le\;&
Cp\sum_{k=1}^n
\Bigl\|
    \|\xi_k\mid\mathcal F_{k-1}\|_{3p,r}
    \|(\Pi_k+M)^{-1/2}\xi_k
        \mid\mathcal F_{k-1}\|_{3p,2}^{2}
\Bigr\|_p
\\
&\quad
+2\|M^{1/2}Z\|_{p,r}.
\end{aligned}
\]

It remains to remove the Gaussian augmentation and to return to the target
\(N(0,\bar\Sigma)\). By the triangle inequality,
\[
\begin{aligned}
\mcl W_{p,r}\bigl(\mathcal L(S_n),N(0,\bar\Sigma)\bigr)
\le\;&
\mcl W_{p,r}\bigl(\mathcal L(S_n),\mathcal L(\widetilde S)\bigr)
\\
&\quad
+
\mcl W_{p,r}\bigl(\mathcal L(\widetilde S),N(0,\bar\Sigma+\Delta_0)\bigr)
\\
&\quad
+
\mcl W_{p,r}\bigl(N(0,\bar\Sigma+\Delta_0),N(0,\bar\Sigma)\bigr).
\end{aligned}
\]
For the first term, using the coupling
\(\widetilde S=S_n+H^{1/2}Z_0\), we have
\[
    \mcl W_{p,r}\bigl(\mathcal L(S_n),\mathcal L(\widetilde S)\bigr)
    \le
    \|H^{1/2}Z_0\|_{p,r}.
\]
The last Gaussian comparison term is bounded by $\|\Delta_0^{1/2} Z\|_{p,r}$.
Combining the preceding estimates proves the theorem.
\end{proof}

\begin{proof}[Proof of \cref{thm:4}]
Following \cite[Proof of Lemma~B.8]{cattaneo2025yurinskii}, fix \(\gamma>0\)
and \(0\le a\le1\) with \(V_k\succeq a\bar\Sigma/n\).
Take \(a=\alpha\) for the elliptic bound and \(a=0\) for the smoothing bound.
For \(Z_0\sim N(0,I_d)\) independent of \(\mathcal F_n\), put
\[
\begin{aligned}
 H_k&:=\bar\Sigma+\gamma^2I_d-\sum_{i=1}^kV_i,\qquad 0\le k\le n,\\
 \tau_a&:=\max\left\{0\le k\le n:
 H_k\succeq\frac{a(n-k)}n\bar\Sigma\right\},\\
 \widetilde S^{(a)}&:=S_{\tau_a}+H_{\tau_a}^{1/2}Z_0.
\end{aligned}
\]
The stopping set is nonempty, since \(H_0-a\bar\Sigma\succeq0\), and
\[
\begin{aligned}
 &\left(H_{k+1}-\frac{a(n-k-1)}n\bar\Sigma\right)
  -\left(H_k-\frac{a(n-k)}n\bar\Sigma\right)
 =-V_{k+1}+\frac an\bar\Sigma\preceq0,\\
 &\{k\le\tau_a\}
 =\left\{H_k\succeq\frac{a(n-k)}n\bar\Sigma\right\}\in\mathcal F_{k-1},\\
 &H_{\tau_a}\succeq0,\qquad
 \sum_{k=1}^nV_k\mathbf1_{\{k\le\tau_a\}}+H_{\tau_a}
 =\bar\Sigma+\gamma^2I_d.
\end{aligned}
\]
Thus \(\tau_a\) is a bounded stopping time. Adjoining \(Z_0\) at time \(n+1\)
makes
\[
 \xi_1\mathbf1_{\{1\le\tau_a\}},\ldots,
 \xi_n\mathbf1_{\{n\le\tau_a\}},H_{\tau_a}^{1/2}Z_0
\]
a martingale difference sequence with sum \(\widetilde S^{(a)}\).
At stopped steps and at the terminal conditional Gaussian step, Lindeberg swapping errors are zero.

Write \(\Delta=\Sigma-\bar\Sigma\). Monotonicity and Markov's inequality give
\begin{equation}\label{eq:thm23-tau-prob}
\begin{aligned}
 \{\tau_a<n\}
 &=\{\gamma^2I_d-\Delta\not\succeq0\}
 \subseteq\{\|\Delta\|_2>\gamma^2\},\\
 \mathbb P(\tau_a<n)^{2/(3p)}
 &\le\gamma^{-2}\|\Delta\|_{3p/2,2}.
\end{aligned}
\end{equation}
Independent Gaussian addition and the triangle inequality yield
\[
\begin{aligned}
 &\mcl W_{p,r}\bigl(\mcl L(S_n),N(0,\bar\Sigma)\bigr)\\
 &\le\mcl W_{p,r}\bigl(\mcl L(\widetilde S^{(a)}),
                         N(0,\bar\Sigma+\gamma^2I_d)\bigr)\\
 &\quad+\|S_n-S_{\tau_a}\|_{p,r}
       +\|H_{\tau_a}^{1/2}Z_0\|_{p,r}+\gamma\|Z\|_{p,r}.
\end{aligned}
\]
By bounded optional sampling, conditional Jensen and H\"older,
\begin{equation}\label{eq:thm23-tail-bound}
\begin{aligned}
 S_{\tau_a}&=\E(S_n\mid\mathcal F_{\tau_a}),\qquad
 \|S_n-S_{\tau_a}\|_{3p,r}\le2\|S_n\|_{3p,r},\\
 \|S_n-S_{\tau_a}\|_{p,r}
 &\le\|S_n-S_{\tau_a}\|_{3p,r}\mathbb P(\tau_a<n)^{2/(3p)}\\
 &\le2\|S_n\|_{3p,r}\gamma^{-2}\|\Delta\|_{3p/2,2}.
\end{aligned}
\end{equation}
Conditionally on \(\mathcal F_n\), independent Gaussian addition and the
convexity of \(|\cdot|_r^p\) give covariance monotonicity for every \(r\).
On \(\{\tau_a=n\}\),
\[
\begin{aligned}
 0\preceq H_{\tau_a}
 &=\gamma^2I_d-\Delta\preceq(\gamma^2+\|\Delta\|_2)I_d,\\
 \|\mathbf1_{\{\tau_a=n\}}H_{\tau_a}^{1/2}Z_0\|_{p,r}
 &\le\|Z\|_{p,r}
 \left\|\mathbf1_{\{\tau_a=n\}}(\gamma^2+\|\Delta\|_2)^{1/2}\right\|_p\\
 &\le\|Z\|_{p,r}
 \left\{\gamma+(\E\|\Delta\|_2^{p/2})^{1/p}\right\}.
\end{aligned}
\]
On \(\{\tau_a<n\}\), \(0\preceq H_{\tau_a}\preceq\bar\Sigma+\gamma^2I_d\).
Since \(\mathbb P(\tau_a<n)^{1/p}\le
\mathbb P(\tau_a<n)^{2/(3p)}\le1\), \cref{eq:thm23-tau-prob} gives
\[
\begin{aligned}
 \|\mathbf1_{\{\tau_a<n\}}H_{\tau_a}^{1/2}Z_0\|_{p,r}
 &\le\mathbb P(\tau_a<n)^{1/p}
 \left\{\|\bar\Sigma^{1/2}Z\|_{p,r}+\gamma\|Z\|_{p,r}\right\}\\
 &\le\gamma^{-2}\|\bar\Sigma^{1/2}Z\|_{p,r}\|\Delta\|_{3p/2,2}
       +\gamma\|Z\|_{p,r}.
\end{aligned}
\]
Combining these estimates,
\begin{equation}\label{eq:thm23-comparison}
\begin{aligned}
 &\|S_n-S_{\tau_a}\|_{p,r}
       +\|H_{\tau_a}^{1/2}Z_0\|_{p,r}+\gamma\|Z\|_{p,r}\\
 &\le C\|Z\|_{p,r}(\E\|\Delta\|_2^{p/2})^{1/p}
       +C\gamma\|Z\|_{p,r}
       +C\gamma^{-2}R_{p,r}\|\Delta\|_{3p/2,2}.
\end{aligned}
\end{equation}

For \(a=\alpha\), the augmented tail covariance on \(\{k\le\tau_\alpha\}\) is
\[
\begin{aligned}
 H_{k-1}&\succeq\frac{\alpha(n-k+1)}n\bar\Sigma\succ0,\\
 \|H_{k-1}^{-1/2}\xi_k\mid\mathcal F_{k-1}\|_{3p,2}^2
 &\le\frac n{\alpha(n-k+1)}
       \|\bar\Sigma^{-1/2}\xi_k\mid\mathcal F_{k-1}\|_{3p,2}^2.
\end{aligned}
\]
The conditional Lindeberg proof of \cref{thm:3}, with
\(M=0\), therefore gives
\begin{equation}\label{eq:thm23-elliptic-stopped-bound}
\begin{aligned}
 &\mcl W_{p,r}\bigl(\mcl L(\widetilde S^{(\alpha)}),
                   N(0,\bar\Sigma+\gamma^2I_d)\bigr)\\
 &\le\frac{Cp}{\alpha}\sum_{k=1}^n\frac n{n-k+1}
 \left\|\|\xi_k\mid\mathcal F_{k-1}\|_{3p,r}
 \|\bar\Sigma^{-1/2}\xi_k\mid\mathcal F_{k-1}\|_{3p,2}^2\right\|_p.
\end{aligned}
\end{equation}
For \(a=0\), predictable stopping only decreases \(A_{p,r}\).
The proof of \cref{cor:1}, with \(M=\sigma^2I_d\succ0\), yields
\begin{equation}\label{eq:thm23-modified-smoothing}
\begin{aligned}
 &\mcl W_{p,r}\bigl(\mcl L(\widetilde S^{(0)}),
                   N(0,\bar\Sigma+\gamma^2I_d)\bigr)\\
 &\le\inf_{\sigma>0}
 \left\{\frac{CpA_{p,r}}{\sigma^2}+2\sigma\|Z\|_{p,r}\right\}
 \le C(pA_{p,r})^{1/3}\|Z\|_{p,r}^{2/3}.
\end{aligned}
\end{equation}
Combine each auxiliary bound with \cref{eq:thm23-comparison}.
If \(R_{p,r}\|\Delta\|_{3p/2,2}>0\), choose
\[
\begin{aligned}
 \gamma&=\left(\frac{R_{p,r}\|\Delta\|_{3p/2,2}}{\|Z\|_{p,r}}\right)^{1/3},\\
 \gamma\|Z\|_{p,r}+\gamma^{-2}R_{p,r}\|\Delta\|_{3p/2,2}
 &=2\|Z\|_{p,r}^{2/3}
       \left(R_{p,r}\|\Delta\|_{3p/2,2}\right)^{1/3}.
\end{aligned}
\]
For a zero product, let \(\gamma\downarrow0\).
The remaining lower-moment term $\|Z\|_{p,r}(\E\|\Delta\|_2^{p/2})^{1/p}$ in \cref{eq:thm23-comparison} is removed by the following two bounds:
\[
\begin{aligned}
 &\text{if }\ \|Z\|_{p,r}\|\Delta\|_{3p/2,2}^{1/2}\le R_{p,r},\\
 &\qquad
 \|Z\|_{p,r}(\E\|\Delta\|_2^{p/2})^{1/p}
 \le\|Z\|_{p,r}\|\Delta\|_{3p/2,2}^{1/2}
 \le\|Z\|_{p,r}^{2/3}
       \left(R_{p,r}\|\Delta\|_{3p/2,2}\right)^{1/3};\\[2pt]
 &\text{if }\ \|Z\|_{p,r}\|\Delta\|_{3p/2,2}^{1/2}>R_{p,r},\\
 &\qquad
 \mcl W_{p,r}\bigl(\mcl L(S_n),N(0,\bar\Sigma)\bigr)
 \le\|S_n\|_{p,r}+\|\bar\Sigma^{1/2}Z\|_{p,r}
 \le R_{p,r}\\
 &\qquad\hspace{3em}
 \le\|Z\|_{p,r}^{2/3}
       \left(R_{p,r}\|\Delta\|_{3p/2,2}\right)^{1/3}.
\end{aligned}
\]
The second bound uses an independent coupling. 
Finally, combining \cref{eq:thm23-elliptic-stopped-bound} and \cref{eq:thm23-modified-smoothing} with the above upper bounds on \cref{eq:thm23-comparison} proves \cref{eq:thm23-elliptic,eq:thm23-smoothing}.
\end{proof}

\section{Lemmas and their proofs}\label{sec:lem}

\begin{proof}[Proof of \cref{lem:gaussian-norm-bound}]
Let \(A=\Sigma^{1/2}\), and let \(Z\sim N(0,I_d)\). Then
\(X\stackrel{d}{=}AZ\).

We first consider \(1\le r<\infty\). Set
\[
S_r:=
\left(\sum_{j=1}^d\Sigma_{jj}^{r/2}\right)^{1/r},
\qquad
T_r:=\|A\|_{2\to r}.
\]
Define
\[
f(z):=|Az|_r,\qquad z\in\mathbb R^d.
\]
Then \(f\) is \(T_r\)-Lipschitz with respect to the Euclidean norm. Indeed, for all
\(z,z'\in\mathbb R^d\),
\[
|f(z)-f(z')|
\le |A(z-z')|_r
\le \|A\|_{2\to r}|z-z'|_2
=
T_r|z-z'|_2.
\]
If \(T_r=0\), then \(AZ=0\) almost surely and the desired estimate is trivial. Otherwise,
by the Gaussian concentration inequality for Lipschitz functions,
\[
\mathbb P\{f(Z)-\mathbb Ef(Z)>t\}
\le
\exp\left(-\frac{t^2}{2T_r^2}\right),
\qquad t>0.
\]
Therefore, for \(p\ge1\),
\[
\begin{aligned}
\mathbb E\left[(f(Z)-\mathbb Ef(Z))_+^p\right]
&=
\int_0^\infty
p t^{p-1}
\mathbb P\{f(Z)-\mathbb Ef(Z)>t\}\,dt        \\
&\le
\int_0^\infty
p t^{p-1}
\exp\left(-\frac{t^2}{2T_r^2}\right)\,dt     \\
&=
2^{p/2}\Gamma(p/2+1)T_r^p.
\end{aligned}
\]
Using the standard bound
\[
\left(2^{p/2}\Gamma(p/2+1)\right)^{1/p}\le C\sqrt p,
\]
we obtain
\[
\|(f(Z)-\mathbb Ef(Z))_+\|_p
\le
C\sqrt p\,T_r.
\]
Since
\[
f(Z)\le \mathbb Ef(Z)+(f(Z)-\mathbb Ef(Z))_+,
\]
we have
\ben{\label{eq:concentration}
\|X\|_{p,r}
=
\|f(Z)\|_p
\le
\mathbb Ef(Z)+C\sqrt p\,T_r.
}

It remains to bound \(\mathbb Ef(Z)\). By Jensen's inequality,
\[
\mathbb Ef(Z)
=
\mathbb E|X|_r
\le
\left(\mathbb E|X|_r^r\right)^{1/r}.
\]
Since \(X_j\sim N(0,\Sigma_{jj})\),
\[
\mathbb E|X|_r^r
=
\sum_{j=1}^d\mathbb E|X_j|^r
=
\mathbb E|g|^r
\sum_{j=1}^d\Sigma_{jj}^{r/2},
\]
where \(g\sim N(0,1)\). Hence
\[
\mathbb Ef(Z)
\le
\|g\|_r
\left(\sum_{j=1}^d\Sigma_{jj}^{r/2}\right)^{1/r}.
\]
Using the standard Gaussian moment bound \(\|g\|_r\le C\sqrt r\), we get
\[
\mathbb Ef(Z)
\le
C\sqrt r
\left(\sum_{j=1}^d\Sigma_{jj}^{r/2}\right)^{1/r}.
\]
Combining this with the concentration estimate gives
\[
\|X\|_{p,r}
\le
C\sqrt r
\left(\sum_{j=1}^d\Sigma_{jj}^{r/2}\right)^{1/r}
+
C\sqrt p\,\|\Sigma^{1/2}\|_{2\to r}.
\]

We now consider \(r=\infty\). Put
\[
\sigma^2:=\max_{1\le j\le d}\Sigma_{jj}.
\]
For every \(t>0\),
\[
\begin{aligned}
\mathbb E\|X\|_\infty
&=
\mathbb E\max_{1\le j\le d}|X_j|                                      \\
&\le
t\log\sum_{j=1}^d
\mathbb E\left(e^{X_j/t}+e^{-X_j/t}\right)                              \\
&=
t\log\sum_{j=1}^d
2\exp\left(\frac{\Sigma_{jj}}{2t^2}\right)                               \\
&\le
t\log(2d)+\frac{\sigma^2}{2t}.
\end{aligned}
\]
Choosing \(t=\sigma/\sqrt{2\log(2d)}\) yields
\[
\mathbb E\|X\|_\infty
\le
\sigma\sqrt{2\log(2d)}.
\]
Also, with $e_j$ denoting the $j$th coordinate vector,
\[
\|A\|_{2\to\infty}
=
\max_{1\le j\le d}|A e_j|_2=\max_{1\le j\le d}\sqrt{e_j^\top \Sigma e_j}
=
\sqrt{\max_{1\le j\le d}\Sigma_{jj}}
=
\sigma.
\]
Applying the same concentration argument in \cref{eq:concentration} to
$f(z)=|Az|_\infty$
gives
\[
\|(f(Z)-\mathbb Ef(Z))_+\|_p
\le
C\sqrt p\,\|A\|_{2\to\infty}
=
C\sqrt p\,\sigma.
\]
Therefore,
\[
\|X\|_{p,\infty}
\le
\mathbb E\|X\|_\infty+C\sqrt p\,\sigma
\le
C\sigma\left(\sqrt{\log(2d)}+\sqrt p\right).
\]
This proves the lemma.
\end{proof}

The following lemma is the starting point for bounding the $\mcl{W}_{p,r}$ distance. 
The usual case $r=2$ and $\Sigma=I_d$ used in \cite{LeNoPe15} and \cite{bonis2020stein} follows from the Otto-Villani theorem [\cite{otto2000generalization}].
Our proof for the general case follows from a modification of \cite[Proposition~3.1]{fang2024sharp}.

\begin{lemma}\label{lem:follmer}
Let $p\geq 1$, $r\in [1,\infty]$, and
let $\mu$ be a probability measure on $\mathbb R^d$ with finite \(p\)-th moment. Let $\Sigma$ be a symmetric, positive definite matrix.
We can construct two random vectors $W\sim\mu$ and $Z\sim N(0,I_d)$ defined on a common probability space such that
\[
\||W-\Sigma^{1/2} Z|_r\|_{L^p}\leq 2\int_0^\infty \|\rho_t(F_t)\|_{p,r} dt,
\]
where $Z_2\sim N(0, I_d)$ is independent of $W$ and
\[
F_t = e^{-t} W + \sqrt{1 - e^{-2t}} \Sigma^{1/2} Z_2, \quad \rho_t(F_t) = e^{-t} \mathbb{E} \left[\left. W - \frac{1}{\sqrt{e^{2t}-1}} \Sigma^{1/2} Z_2 \,\right|\, F_t\right].
\]
\end{lemma}

\begin{proof}
We follow the Föllmer-process argument used in
\cite[Proposition~3.1]{fang2024sharp}, which is a constructive version of the
score-based transportation bound of \cite{otto2000generalization}.

Let \(W_0\sim\mu\), let \(Z_2\sim N(0,I_d)\) be independent of \(W_0\), and put
\[
    \bar W:=\Sigma^{-1/2}W_0 .
\]
For \(s\in(0,1)\), define
\[
    X[s]:=\sqrt{s}\,\bar W+\sqrt{1-s}\,Z_2,
\]
and let \(f_s\) be the density of \(X[s]\) with respect to the standard Gaussian
measure \(N(0,I_d)\). For \(s\in(0,1)\), this density is smooth and strictly
positive.

Let \(X=(X_s)_{s\in[0,1]}\) be the Föllmer process associated with the law of
\(\bar W\). Equivalently, \(X_1\) has the same distribution as \(\bar W\), and
conditionally on \(X_1\), the process \(X\) is a \(d\)-dimensional Brownian
bridge from \(0\) to \(X_1\) on \([0,1]\). By
\cite[Theorem~2.1 and Proposition~2.5]{koike2026note}, we have, almost surely,
\[
    \int_0^1
    \frac{|\nabla\log f_u(X_u/\sqrt u)|_2}{\sqrt u}\,du<\infty,
\]
and
\[
    B_s
    :=
    X_s-\int_0^s
    \frac{\nabla \log f_u(X_u/\sqrt u)}{\sqrt u}\,du,
    \qquad 0\le s\le 1,
\]
is a standard Brownian motion. Hence
\[
    W:=\Sigma^{1/2}X_1\sim\mu,
    \qquad
    \Sigma^{1/2}Z:=\Sigma^{1/2}B_1\sim N(0,\Sigma),
\]
and
\[
    W-\Sigma^{1/2}Z
    =
    \int_0^1
    \frac{\Sigma^{1/2}\nabla\log f_u(X_u/\sqrt u)}{\sqrt u}\,du .
\]
This implies
\[
\begin{aligned}
    \big\||W-\Sigma^{1/2}Z|_r\big\|_{L^p}
    &\le
    \int_0^1
    \frac{
    \big\|\Sigma^{1/2}\nabla\log f_u(X_u/\sqrt u)\big\|_{p,r}
    }{\sqrt u}\,du .
\end{aligned}
\]
By the construction of the Föllmer process, \(X_u/\sqrt u\) has the same law as
\(X[u]\). Thus
\[
    \big\|\Sigma^{1/2}\nabla\log f_u(X_u/\sqrt u)\big\|_{p,r}
    =
    \big\|\Sigma^{1/2}\nabla\log f_u(X[u])\big\|_{p,r}.
\]
By the score representation for Gaussian convolutions in
\cite[Lemma~2]{bonis2020stein},
\[
    \nabla\log f_u(X[u])
    =
    \sqrt u\,
    \E\left[
        \bar W-\sqrt{\frac{u}{1-u}}\,Z_2
        \,\middle|\, X[u]
    \right].
\]
Multiplying by \(\Sigma^{1/2}\), we obtain
\[
    \Sigma^{1/2}\nabla\log f_u(X[u])
    =
    \sqrt u\,
    \E\left[
        W_0-\sqrt{\frac{u}{1-u}}\,\Sigma^{1/2}Z_2
        \,\middle|\, \Sigma^{1/2}X[u]
    \right].
\]
Now make the change of variables
\[
    u=e^{-2t}, \qquad t\in[0,\infty).
\]
Then
\[
    \Sigma^{1/2}X[u]
    =
    e^{-t}W_0+\sqrt{1-e^{-2t}}\,\Sigma^{1/2}Z_2
    =
    F_t,
\]
and
\[
    \sqrt{\frac{u}{1-u}}
    =
    \frac{1}{\sqrt{e^{2t}-1}}.
\]
Consequently,
\[
    \Sigma^{1/2}\nabla\log f_{e^{-2t}}(X[e^{-2t}])
    =
    e^{-t}
    \E\left[
        W_0-\frac{1}{\sqrt{e^{2t}-1}}\Sigma^{1/2}Z_2
        \,\middle|\, F_t
    \right]
    =
    \rho_t(F_t).
\]
Hence
\[
\begin{aligned}
    \big\||W-\Sigma^{1/2}Z|_r\big\|_{L^p}
    &\le
    \int_0^1
    \frac{
    \big\|\Sigma^{1/2}\nabla\log f_u(X[u])\big\|_{p,r}
    }{\sqrt u}\,du  \\
    &=
    2\int_0^\infty
    e^{-t}\|\rho_t(F_t)\|_{p,r}\,dt  \\
    &\le
    2\int_0^\infty
    \|\rho_t(F_t)\|_{p,r}\,dt .
\end{aligned}
\]

\end{proof}

\begin{lemma}\label{lem:1}
For \(\tau_s\) defined in \cref{eq:taus}, the conditional expectation satisfies
\[
\mathbb{E}\!\left(
\tau_s
\,\middle|\,
Z+e^{-s}\widetilde I_{d,s}^{-1/2}\widetilde\xi_{\Sigma,s}
\right)=0.
\]
\end{lemma}

\begin{proof}
For notational simplicity, write
\[
u_s:=\widetilde I_{d,s}^{-1/2}\widetilde\xi_{\Sigma,s},
\qquad
a_s:=\sqrt{1-e^{-2s}}\,u_s,
\qquad
\widetilde F_s:=Z+e^{-s}u_s.
\]
Because \(u_s\) is bounded by the truncation in
\(\widetilde\xi_{\Sigma,s}\), the series
\be{
1+\sum_{k=1}^{\infty}\frac{e^{-ks}}{k!}
\left\langle u_s^{\otimes k},
\frac{\nabla^k\phi(Z)}{\phi(Z)}\right\rangle
}
converges in $L^2$.

Given a bounded measurable function \(h:\mathbb R^d\to\mathbb R\) and \(s>0\),
define
\[
T_sh(x)=\mathbb E\left[
h\left(e^{-s}x+\sqrt{1-e^{-2s}}Z\right)
\right],
\qquad x\in\mathbb R^d.
\]
We have
\[
\nabla^kT_sh(x)
=
\frac{(-1)^k}{(e^{2s}-1)^{k/2}}
\mathbb E\left[
\frac{\nabla^k\phi(Z)}{\phi(Z)}
h\left(e^{-s}x+\sqrt{1-e^{-2s}}Z\right)
\right],
\]
and hence
\[
\begin{aligned}
0
&=\mathbb E[\xi T_sh(0)] \\
&=\mathbb E\left[\xi\left(
T_sh(a_s)
+\sum_{k=1}^{\infty}\frac{(-1)^k}{k!}
\left\langle a_s^{\otimes k},\nabla^kT_sh(a_s)\right\rangle
\right)\right] \\
&=\mathbb E\left[
h\left(e^{-s}a_s+\sqrt{1-e^{-2s}}Z\right)\xi
\left(
1+\sum_{k=1}^{\infty}\frac{e^{-ks}}{k!}
\left\langle u_s^{\otimes k},
\frac{\nabla^k\phi(Z)}{\phi(Z)}\right\rangle
\right)\right] \\
&=\mathbb E\left[
h\left(\sqrt{1-e^{-2s}}\,\widetilde F_s\right)\tau_s
\right].
\end{aligned}
\]
Since multiplication by the nonzero scalar \(\sqrt{1-e^{-2s}}\) is invertible,
the preceding identity, valid for every bounded measurable \(h\), proves
\(\mathbb E(\tau_s\mid\widetilde F_s)=0\).
\end{proof}

The following lemma quantifies the change in a conditional expectation caused by an
additive perturbation of the conditioning random vector. It will be used repeatedly in
the proof of \cref{prop:1}.

\begin{lemma}\label{lem:2}
Let $p\geq 1$ and $r\in [1,\infty]$.
Let $X,Y,Z$ be random vectors in $\mathbb{R}^d$ such that $Z \sim N(0,I_d)$ and $Z$ is independent of $(X,Y)$.
For a measurable function $H:\mathbb{R}^{3d}\to\mathbb{R}^d$ that is continuously differentiable in the $Z$-argument, assume that all terms in the estimates below are finite, and define
\[
G_t := Z+X+tY,\qquad t\in[0,1].
\]
Then
\besn{\label{eq:integralbound}
&\big\|\mathbb{E}[H(X,Y,Z)\mid Z+X+Y]
-\mathbb{E}[H(X,Y,Z)\mid Z+X]\big\|_{p,r} \\
&\le
\int_0^1 \Big(
\big\|\mathbb{E}\!\big[\tfrac{\partial H}{\partial Z}(X,Y,Z)\,\big|\,G_t\big]\,Y\big\|_{p,r}\\[2pt]
&\qquad\qquad
+\big\|\mathbb{E}\!\big[\tfrac{\partial H}{\partial Z}(X,Y,Z)\,Y\,\big|\,G_t\big]\big\|_{p,r}\\[2pt]
&\qquad\qquad
+\big\|\mathbb{E}\!\big(H(X,Y,Z)\,\langle Z,Y\rangle\,\big|\,G_t\big)\big\|_{p,r}\\[2pt]
&\qquad\qquad
+\big\|\mathbb{E}\!\big(H(X,Y,Z)Z^{\top}\,\big|\,G_t\big)\,Y\big\|_{p,r}\\[2pt]
&\qquad\qquad
+\big\|\mathbb{E}\!\big(\langle Y,Z\rangle\,\big|\,G_t\big)\,
        \mathbb{E}\!\big(H(X,Y,Z)\,\big|\,G_t\big)\big\|_{p,r}\\[2pt]
&\qquad\qquad
+\big\|\langle Y,\mathbb{E}[Z\,|\,G_t]\rangle\,
        \mathbb{E}\!\big(H(X,Y,Z)\,\big|\,G_t\big)\big\|_{p,r}
\Big)\,dt .
}
Moreover, for any
\[
\rho,a,c,\alpha_0,\beta_0,\alpha_1,\beta_1,\gamma_1,
\alpha_2,\beta_2,\gamma_2\in[1,\infty]
\]
satisfying
\[
\frac1p=\frac1a+\frac1c,\qquad
\frac1p=\frac1{\alpha_0}+\frac1{\beta_0},\qquad
\frac1p=\frac1{\alpha_1}+\frac1{\beta_1}+\frac1{\gamma_1},\qquad
\frac1p=\frac1{\alpha_2}+\frac1{\beta_2}+\frac1{\gamma_2},
\]
we have
\[
\begin{aligned}
&\left\|\mathbb{E}[H(X,Y,Z)\mid Z+X+Y]
-\mathbb{E}[H(X,Y,Z)\mid Z+X]\right\|_{p,r} \\
&\le
2\left\|
\frac{\partial H}{\partial Z}(X,Y,Z)
\right\|_{a,\rho\to r}
\|Y\|_{c,\rho} \\
&\quad+
4\|H(X,Y,Z)\|_{\alpha_0,r}
\|Y\|_{\beta_0,2}\|g\|_{\beta_0} \\
&\quad+
4\|H(X,Y,Z)\|_{\alpha_1,r}
\|X\|_{\beta_1,2}\|Y\|_{\gamma_1,2} \\
&\quad+
2\|H(X,Y,Z)\|_{\alpha_2,r}
\|Y\|_{\beta_2,2}\|Y\|_{\gamma_2,2},
\end{aligned}
\]
where \(\|g\|_{\beta_0}\) denotes the
\(L^{\beta_0}\)-norm of a one-dimensional standard Gaussian.

The same conclusions hold if $H$ also depends on an auxiliary random element $U$
such that $Z$ is independent of $(X,Y,U)$.
\end{lemma}

\begin{proof}
The proof of the lemma is inspired by \cite[Lemma~6.3]{bonis2020stein}.
We give the proof without $U$. For the auxiliary-variable version, one replaces the
joint law of $(X,Y)$ below by that of $(X,Y,U)$ and carries the $U$-coordinate as an
unchanged argument of $H$ throughout.
As another remark, the differentiations below can be justified under the stated integrability assumptions and we omit the technical discussions.

Write
\[
\begin{aligned}
&\mathbb{E}[H(X,Y,Z)\mid Z+X+Y]
-\mathbb{E}[H(X,Y,Z)\mid Z+X] \\
&\qquad
=
\int_0^1
\frac{d}{dt}\,
\mathbb{E}[H(X,Y,Z)\mid G_t]\,dt .
\end{aligned}
\]

Let $\nu$ be the joint law of $(X,Y)$ on $\mathbb{R}^{2d}$.
For $t\in[0,1]$, define
\[
z_t(x',y'):=Z+(X-x')+t(Y-y')=G_t-x'-ty'.
\]
Set
\[
\begin{aligned}
f(t)
&=
\int
H\!\left(x',y',z_t(x',y')\right)
\phi\!\left(z_t(x',y')\right)\,d\nu(x',y'),\\
g(t)
&=
\int
\phi\!\left(z_t(x',y')\right)\,d\nu(x',y').
\end{aligned}
\]
Then
\[
\frac{f(t)}{g(t)}
=
\mathbb{E}\!\left(H(X,Y,Z)\mid G_t\right).
\]

Differentiating $f$ and $g$ gives
\[
\begin{aligned}
f'(t)
&=
\int
\Big[
\frac{\partial H}{\partial Z}\!\left(x',y',z_t(x',y')\right)(Y-y') \\
&\qquad\qquad
-
H\!\left(x',y',z_t(x',y')\right)
\left\langle z_t(x',y'),Y-y'\right\rangle
\Big]
\phi\!\left(z_t(x',y')\right)\,d\nu(x',y'),
\end{aligned}
\]
and
\[
g'(t)
=
-\int
\left\langle z_t(x',y'),Y-y'\right\rangle
\phi\!\left(z_t(x',y')\right)\,d\nu(x',y').
\]
Hence, by the definition of conditional expectation,
\[
\begin{aligned}
\frac{f'(t)}{g(t)}
&=
\mathbb{E}\!\left(\left.\frac{\partial H}{\partial Z}(X,Y,Z)\right|G_t\right)Y
-
\mathbb{E}\!\left(\left.\frac{\partial H}{\partial Z}(X,Y,Z)Y\right|G_t\right)\\
&\quad
+
\mathbb{E}\!\left(H(X,Y,Z)\langle Z,Y\rangle\mid G_t\right)
-
\mathbb{E}\!\left(H(X,Y,Z)Z^\top\mid G_t\right)Y,
\end{aligned}
\]
and
\[
\frac{g'(t)}{g(t)}
=
\mathbb{E}\!\left(\langle Y,Z\rangle\mid G_t\right)
-
\left\langle Y,\mathbb{E}\!\left(Z\mid G_t\right)\right\rangle .
\]
Therefore,
\[
\begin{aligned}
\frac{d}{dt}\mathbb{E}\!\left(H(X,Y,Z)\mid G_t\right)
&=
\left(\frac{f}{g}\right)'(t)\\
&=
\mathbb{E}\!\left(\left.\frac{\partial H}{\partial Z}(X,Y,Z)\right|G_t\right)Y
-
\mathbb{E}\!\left(\left.\frac{\partial H}{\partial Z}(X,Y,Z)Y\right|G_t\right)\\
&\quad
+
\mathbb{E}\!\left(H(X,Y,Z)\langle Z,Y\rangle\mid G_t\right)
-
\mathbb{E}\!\left(H(X,Y,Z)Z^\top\mid G_t\right)Y\\
&\quad
-
\mathbb{E}\!\left(\langle Y,Z\rangle\mid G_t\right)
\mathbb{E}\!\left(H(X,Y,Z)\mid G_t\right)\\
&\quad
+
\left\langle Y,\mathbb{E}\!\left(Z\mid G_t\right)\right\rangle
\mathbb{E}\!\left(H(X,Y,Z)\mid G_t\right).
\end{aligned}
\]
Consequently, by the integral Minkowski inequality and the triangle inequality,
\[
\begin{aligned}
&\big\|\mathbb{E}[H(X,Y,Z)\mid Z+X+Y]
-\mathbb{E}[H(X,Y,Z)\mid Z+X]\big\|_{p,r}\\
&\le
\int_0^1
\left\|
\frac{d}{dt}\mathbb{E}\!\left(H(X,Y,Z)\mid G_t\right)
\right\|_{p,r}\,dt\\
&\le
\int_0^1 \Big(
\big\|\mathbb{E}\!\big[\tfrac{\partial H}{\partial Z}(X,Y,Z)\,\big|\,G_t\big]\,Y\big\|_{p,r}
+
\big\|\mathbb{E}\!\big[\tfrac{\partial H}{\partial Z}(X,Y,Z)Y\,\big|\,G_t\big]\big\|_{p,r}\\
&\qquad\quad
+
\big\|\mathbb{E}\!\big(H(X,Y,Z)\langle Z,Y\rangle\,\big|\,G_t\big)\big\|_{p,r}
+
\big\|\mathbb{E}\!\big(H(X,Y,Z)Z^\top\,\big|\,G_t\big)Y\big\|_{p,r}\\
&\qquad\quad
+
\big\|\mathbb{E}\!\big(\langle Y,Z\rangle\,\big|\,G_t\big)
        \mathbb{E}\!\big(H(X,Y,Z)\,\big|\,G_t\big)\big\|_{p,r}\\
&\qquad\quad
+
\big\|\langle Y,\mathbb{E}[Z\,|\,G_t]\rangle
        \mathbb{E}\!\big(H(X,Y,Z)\,\big|\,G_t\big)\big\|_{p,r}
\Big)\,dt .
\end{aligned}
\]
This proves the first assertion.

For the additional estimate, we start from the preceding integral bound \cref{eq:integralbound}.
First, we bound the two terms involving \(\partial_ZH\) using the
\(\rho\to r\) operator norm. By conditional Jensen's inequality and
Hölder's inequality,
\[
\begin{aligned}
&\left\|
\mathbb{E}\!\left[
\left.\frac{\partial H}{\partial Z}(X,Y,Z)\right|G_t
\right]Y
\right\|_{p,r} \\
&\qquad\le
\left\|
\mathbb{E}\!\left[
\left.
\left\|
\frac{\partial H}{\partial Z}(X,Y,Z)
\right\|_{\rho\to r}
\right|G_t
\right]
|Y|_\rho
\right\|_p \\
&\qquad\le
\left\|
\frac{\partial H}{\partial Z}(X,Y,Z)
\right\|_{a,\rho\to r}
\|Y\|_{c,\rho}.
\end{aligned}
\]
Similarly,
\[
\begin{aligned}
\left\|
\mathbb{E}\!\left[
\left.
\frac{\partial H}{\partial Z}(X,Y,Z)Y
\right|G_t
\right]
\right\|_{p,r}
&\le
\left\|
\frac{\partial H}{\partial Z}(X,Y,Z)Y
\right\|_{p,r} \\
&\le
\left\|
\frac{\partial H}{\partial Z}(X,Y,Z)
\right\|_{a,\rho\to r}
\|Y\|_{c,\rho}.
\end{aligned}
\]
Therefore,
\[
\begin{aligned}
&\left\|
\mathbb{E}\!\left[
\left.\frac{\partial H}{\partial Z}(X,Y,Z)\right|G_t
\right]Y
\right\|_{p,r}
+
\left\|
\mathbb{E}\!\left[
\left.
\frac{\partial H}{\partial Z}(X,Y,Z)Y
\right|G_t
\right]
\right\|_{p,r} \\
&\qquad\le
2\left\|
\frac{\partial H}{\partial Z}(X,Y,Z)
\right\|_{a,\rho\to r}
\|Y\|_{c,\rho}.
\end{aligned}
\]

Next, using
\[
Z=G_t-X-tY
\]
and the \(G_t\)-measurability of \(G_t\), we have
\[
\begin{aligned}
\mathbb{E}\!\big(H(X,Y,Z)Z^\top\mid G_t\big)Y
&=
G_t^\top Y\,\mathbb{E}\!\big(H(X,Y,Z)\mid G_t\big) \\
&\quad
-\mathbb{E}\!\big(H(X,Y,Z)X^\top\mid G_t\big)Y \\
&\quad
-t\,\mathbb{E}\!\big(H(X,Y,Z)Y^\top\mid G_t\big)Y,
\end{aligned}
\]
and
\[
\begin{aligned}
Y^\top\mathbb{E}[Z\mid G_t]\,
\mathbb{E}\!\big(H(X,Y,Z)\mid G_t\big)
&=
G_t^\top Y\,\mathbb{E}\!\big(H(X,Y,Z)\mid G_t\big) \\
&\quad
-Y^\top\mathbb{E}[X\mid G_t]\,
\mathbb{E}\!\big(H(X,Y,Z)\mid G_t\big) \\
&\quad
-t\,Y^\top\mathbb{E}[Y\mid G_t]\,
\mathbb{E}\!\big(H(X,Y,Z)\mid G_t\big).
\end{aligned}
\]
After inserting these identities and splitting
\[
G_t^\top Y=Z^\top Y+X^\top Y+tY^\top Y,
\]
we will estimate the remaining terms in \cref{eq:integralbound} by conditional Jensen's inequality and
Hölder's inequality.

First,
\[
\begin{aligned}
&\big\|
\mathbb{E}\!\big(H(X,Y,Z)Y^\top Z\mid G_t\big)
\big\|_{p,r}
+
\big\|
\mathbb{E}\!\big(Y^\top Z\mid G_t\big)
\mathbb{E}\!\big(H(X,Y,Z)\mid G_t\big)
\big\|_{p,r} \\
&\quad
+
2\big\|
Y^\top Z\,\mathbb{E}\!\big(H(X,Y,Z)\mid G_t\big)
\big\|_{p,r} \\
&\qquad\le
4\|H(X,Y,Z)\|_{\alpha_0,r}
\|Y^\top Z\|_{\beta_0}.
\end{aligned}
\]
Since \(Z\) is independent of \(Y\) and \(Z\sim N(0,I_d)\),
\[
\|Y^\top Z\|_{\beta_0}
=
\|Y\|_{\beta_0,2}\|g\|_{\beta_0}.
\]
Hence
\[
\begin{aligned}
&\big\|
\mathbb{E}\!\big(H(X,Y,Z)Y^\top Z\mid G_t\big)
\big\|_{p,r}
+
\big\|
\mathbb{E}\!\big(Y^\top Z\mid G_t\big)
\mathbb{E}\!\big(H(X,Y,Z)\mid G_t\big)
\big\|_{p,r} \\
&\quad
+
2\big\|
Y^\top Z\,\mathbb{E}\!\big(H(X,Y,Z)\mid G_t\big)
\big\|_{p,r} \\
&\qquad\le
4\|H(X,Y,Z)\|_{\alpha_0,r}
\|Y\|_{\beta_0,2}\|g\|_{\beta_0}.
\end{aligned}
\]

Second,
\[
\begin{aligned}
&2\left\|
Y^\top X\,\mathbb{E}\!\left(H(X,Y,Z)\mid G_t\right)
\right\|_{p,r}
+
\left\|
\mathbb{E}\!\left(H(X,Y,Z)X^\top\mid G_t\right)Y
\right\|_{p,r} \\
&\quad
+
\left\|
Y^\top\mathbb{E}[X\mid G_t]\,
\mathbb{E}\!\left(H(X,Y,Z)\mid G_t\right)
\right\|_{p,r} \\
&\qquad\le
4\|H(X,Y,Z)\|_{\alpha_1,r}
\|X\|_{\beta_1,2}
\|Y\|_{\gamma_1,2}.
\end{aligned}
\]

Third,
\[
\begin{aligned}
&2t\left\|
Y^\top Y\,\mathbb{E}\!\left(H(X,Y,Z)\mid G_t\right)
\right\|_{p,r}
+
t\left\|
\mathbb{E}\!\left(H(X,Y,Z)Y^\top\mid G_t\right)Y
\right\|_{p,r} \\
&\quad
+
t\left\|
Y^\top\mathbb{E}[Y\mid G_t]\,
\mathbb{E}\!\left(H(X,Y,Z)\mid G_t\right)
\right\|_{p,r} \\
&\qquad\le
4t\|H(X,Y,Z)\|_{\alpha_2,r}
\|Y\|_{\beta_2,2}
\|Y\|_{\gamma_2,2}.
\end{aligned}
\]

Combining these estimates and integrating over \(t\in[0,1]\), we obtain
\[
\begin{aligned}
&\left\|\mathbb{E}[H(X,Y,Z)\mid Z+X+Y]
-\mathbb{E}[H(X,Y,Z)\mid Z+X]\right\|_{p,r} \\
&\le
\int_0^1
\Bigg[
2\left\|
\frac{\partial H}{\partial Z}(X,Y,Z)
\right\|_{a,\rho\to r}
\|Y\|_{c,\rho} \\
&\qquad\qquad
+
4\|H(X,Y,Z)\|_{\alpha_0,r}
\|Y\|_{\beta_0,2}\|g\|_{\beta_0} \\
&\qquad\qquad
+
4\|H(X,Y,Z)\|_{\alpha_1,r}
\|X\|_{\beta_1,2}\|Y\|_{\gamma_1,2} \\
&\qquad\qquad
+
4t\|H(X,Y,Z)\|_{\alpha_2,r}
\|Y\|_{\beta_2,2}\|Y\|_{\gamma_2,2}
\Bigg]\,dt \\
&=
2\left\|
\frac{\partial H}{\partial Z}(X,Y,Z)
\right\|_{a,\rho\to r}
\|Y\|_{c,\rho} \\
&\quad
+
4\|H(X,Y,Z)\|_{\alpha_0,r}
\|Y\|_{\beta_0,2}\|g\|_{\beta_0} \\
&\quad
+
4\|H(X,Y,Z)\|_{\alpha_1,r}
\|X\|_{\beta_1,2}\|Y\|_{\gamma_1,2} \\
&\quad
+
2\|H(X,Y,Z)\|_{\alpha_2,r}
\|Y\|_{\beta_2,2}\|Y\|_{\gamma_2,2}.
\end{aligned}
\]
This proves the additional assertion.
\end{proof}

The following lemma provides a Rosenthal--Burkholder bound for a
martingale sum. It can be used to bound $\|S_n\|_{3p,r}$ in applying \cref{thm:4}.

\begin{lemma}[Rosenthal--Burkholder inequality]
\label{lem:4}
Let \((\xi_i,\mathcal F_i)_{i=1}^n\) be an \(\mathbb R^d\)-valued martingale
difference sequence. Let \(q\ge2\) and \(r\in[1,\infty]\). Define
\[
    r_\circ:=
    \begin{cases}
        2, & 1\le r<2,\\
        r, & 2\le r\le\infty,
    \end{cases}
    \qquad
    C_r:=
    \begin{cases}
        d^{1/r-1/2}, & 1\le r<2,\\
        1, & 2\le r\le\infty,
    \end{cases}
\]
and
\[
    \kappa_r:=
    \begin{cases}
        1, & 1\le r\le2,\\
        \sqrt{r-1}, & 2<r<\infty,\\
        \sqrt{\log(ed)}, & r=\infty .
    \end{cases}
\]
Then
\[
\begin{aligned}
    \left\|
        \sum_{i=1}^n \xi_i
    \right\|_{q,r}
    \le\;&
    C C_r
    \bigg\{
        q\|\max_{1\le i\le n}|\xi_i|_{r_\circ}\|_q
        \\
    &\qquad\qquad
        +
        \sqrt q\,\kappa_r
        \left\|
            \left(
                \sum_{i=1}^n
                \E\bigl[|\xi_i|_{r_\circ}^2\mid\mathcal F_{i-1}\bigr]
            \right)^{1/2}
        \right\|_q
    \bigg\},
\end{aligned}
\]
where \(C>0\) is a universal constant.
\end{lemma}

\begin{proof}
We first consider the case \(2\le r<\infty\). Apply the
Rosenthal--Burkholder inequality for martingales in \((2,D)\)-smooth Banach
spaces to
\[
    X=(\mathbb R^d,|\cdot|_r).
\]
For \(2\le r<\infty\), this space is
\((2,\sqrt{r-1})\)-smooth; see \cite[Proposition~2.1]{pinelis1994optimum}. Hence, by
\cite[Theorem~4.1]{pinelis1994optimum}, for \(q\ge2\),
\[
\begin{aligned}
    \left\|\sum_{i=1}^n\xi_i\right\|_{q,r}
    &\le
    \left\|\sup_{0\le j\le n}|\sum_{i=1}^j\xi_i|_r\right\|_q
    \\
    &\le
    C
    \bigg\{
        q\|\max_{1\le i\le n}|\xi_i|_r\|_q
        +
        \sqrt{q(r-1)}
        \left\|
            \left(
                \sum_{i=1}^n
                \E[|\xi_i|_r^2\mid\mathcal F_{i-1}]
            \right)^{1/2}
        \right\|_q
    \bigg\}.
\end{aligned}
\]
This proves the claimed bound when \(2\le r<\infty\).

Next consider \(r=\infty\). Let
\[
    s:=2\vee\log d .
\]
Then \(2\le s<\infty\), \(d^{1/s}\le e\), and \(s-1\le C\log(ed)\).
For every \(x\in\mathbb R^d\),
\[
    |x|_\infty\le |x|_s\le d^{1/s}|x|_\infty\le e|x|_\infty.
\]
Applying the already proved finite-\(r\) case with \(r=s\), we obtain
\[
\begin{aligned}
    \left\|
        \sum_{i=1}^n\xi_i
    \right\|_{q,\infty}
    &\le
    \left\|
        \sum_{i=1}^n\xi_i
    \right\|_{q,s}
    \\
    &\le
    C
    \bigg\{
        q\|\max_{1\le i\le n}|\xi_i|_s\|_q
        \\
    &\qquad\qquad
        +
        \sqrt{q(s-1)}
        \left\|
            \left(
                \sum_{i=1}^n
                \E[|\xi_i|_s^2\mid\mathcal F_{i-1}]
            \right)^{1/2}
        \right\|_q
    \bigg\}
    \\
    &\le
    C
    \bigg\{
        q\|\max_{1\le i\le n}|\xi_i|_\infty\|_q
        \\
    &\qquad\qquad
        +
        \sqrt{q\log(ed)}
        \left\|
            \left(
                \sum_{i=1}^n
                \E[|\xi_i|_\infty^2\mid\mathcal F_{i-1}]
            \right)^{1/2}
        \right\|_q
    \bigg\}.
\end{aligned}
\]
This is the desired \(r=\infty\) bound.

Finally, let \(1\le r<2\). Since
\[
    |x|_r\le d^{1/r-1/2}|x|_2,
    \qquad x\in\mathbb R^d,
\]
we have
\[
    \left\|
        \sum_{i=1}^n\xi_i
    \right\|_{q,r}
    \le
    d^{1/r-1/2}
    \left\|
        \sum_{i=1}^n\xi_i
    \right\|_{q,2}.
\]
Applying the already proved \(r=2\) case gives
\[
\begin{aligned}
    \left\|
        \sum_{i=1}^n\xi_i
    \right\|_{q,r}
    \le\;&
    C d^{1/r-1/2}
    \bigg\{
        q\|\max_{1\le i\le n}|\xi_i|_2\|_q
        \\
    &\qquad\qquad
        +
        \sqrt q
        \left\|
            \left(
                \sum_{i=1}^n
                \E[|\xi_i|_2^2\mid\mathcal F_{i-1}]
            \right)^{1/2}
        \right\|_q
    \bigg\}.
\end{aligned}
\]
This is the asserted bound for \(1\le r<2\). The proof is complete.
\end{proof}

\section{Proofs of the applications}\label{sec:proofapp}

\begin{proof}[Proof of \cref{thm:tail-cor22}]
Fix $\gamma>0$. Let \(\widetilde S:=\widetilde S^{(0)}\) and
\(\tau:=\tau_0\) be the stopped Gaussian augmentation and stopping time
constructed in the proof of \cref{thm:4} with parameter $\gamma$.
Its total predictable quadratic variation is
$\bar\Sigma+\gamma^2I_d$. Conditional on $\mathcal F_n$, its terminal
increment and its Gaussian replacement both have law $N(0,H_{\tau_0})$,
so their Lindeberg swapping error is zero. Moreover, stopping only
decreases the corresponding $A_{p,r}$ term. Therefore, by the proof of
\cref{cor:1},
\[
\mcl W_{p,r}\!\left(
\mathcal L(\widetilde S),N(0,\bar\Sigma+\gamma^2I_d)
\right)
\le
C(pA_{p,r})^{1/3}\|Z\|_{p,r}^{2/3}.
\]
Hence, possibly enlarging the probability space, there exists
$\widetilde T\sim N(0,\bar\Sigma+\gamma^2I_d)$ such that
\[
\mathbb P\left(|\widetilde S-\widetilde T|_r>\eta\right)
\le
\left[
\frac{
C(pA_{p,r})^{1/3}\|Z\|_{p,r}^{2/3}
}{\eta}
\right]^p .
\]
Using a regular conditional distribution for the Gaussian decomposition
\(N(0,\bar\Sigma+\gamma^2I_d)=N(0,\bar\Sigma)*N(0,\gamma^2I_d)\), we may
extend this coupling further so that
\[
    \widetilde T=T+\gamma Z_1,
    \qquad
    T\sim N(0,\bar\Sigma),
    \qquad
    Z_1\sim N(0,I_d),
\]
with \(T\) and \(Z_1\) independent, without changing the joint law of
\((\widetilde S,\widetilde T)\).

By the same arguments comparing the stopped problem with the original one (cf. \cite[Lemma~B.8 and the proof of Proposition~2.1]
{cattaneo2025yurinskii}), we have the following comparison. 
Note that
\(\{\tau<n\}\subseteq\{\|\Delta\|_2>\gamma^2\}\), whereas on
\(\{\tau=n\}\),
\[
S_n-\widetilde S=-(\gamma^2I_d-\Delta)^{1/2}Z_0.
\]
On \(\{\tau=n\}\), the conditional Gaussian bound
\[
\mathbb E\!\left[
|(\gamma^2I_d-\Delta)^{1/2}Z_0|_r
\,\middle|\,\mathcal F_n
\right]
\le
\|Z\|_{1,r}\|\gamma^2I_d-\Delta\|_2^{1/2}
\]
holds. Consequently, Markov's inequality and Cauchy--Schwarz give
\[
\begin{aligned}
\mathbb P\left(|S_n-\widetilde S|_r>\eta\right)
&\le
\mathbb P(\tau<n)
+
\frac1\eta
\mathbb E\!\left[
\mathbf 1_{\{\tau=n\}}
|H_\tau^{1/2}Z_0|_r
\right] \\
&\le
\frac{\mathbb E\|\Delta\|_2}{\gamma^2}
+
\frac{\gamma\|Z\|_{1,r}}{\eta}
+
\frac{\|Z\|_{1,r}\{\mathbb E\|\Delta\|_2\}^{1/2}}{\eta}.
\end{aligned}
\]
Moreover,
\[
\mathbb P\left(|\widetilde T-T|_r>\eta\right)
\le
\frac{\gamma\|Z\|_{1,r}}{\eta}.
\]
Combining these two bounds yields
\[
\mathbb P\left(|S_n-\widetilde S|_r>\eta\right)
+
\mathbb P\left(|\widetilde T-T|_r>\eta\right)
\le
\frac{\mathbb E\|\Delta\|_2}{\gamma^2}
+
\frac{2\gamma\|Z\|_{1,r}}{\eta}
+
\frac{\|Z\|_{1,r}\{\mathbb E\|\Delta\|_2\}^{1/2}}{\eta}.
\]
Since
\[
        |S_n-T|_r
        \le
        |S_n-\widetilde S|_r
        +
        |\widetilde S-\widetilde T|_r
        +
        |\widetilde T-T|_r ,
\]
we have
\[
\begin{aligned}
\mathbb P\left(|S_n-T|_r>3\eta\right)
&\le
\left[
\frac{
C(pA_{p,r})^{1/3}\|Z\|_{p,r}^{2/3}
}{\eta}
\right]^p                                                     \\
&\quad+
\frac{\mathbb E\|\Delta\|_2}{\gamma^2}
+
\frac{2\gamma\|Z\|_{1,r}}{\eta}
+
\frac{\|Z\|_{1,r}\{\mathbb E\|\Delta\|_2\}^{1/2}}{\eta}.
\end{aligned}
\]
The proof is complete by optimizing over $\gamma$.
\end{proof}

\begin{proof}[Proof of \cref{thm:MDmartingale}]
We first derive the Wasserstein estimate needed below. For
\(n\ge\max\{1,2\gamma^2\}\), put \(m_k=n-k+1\). The assumptions give
\[
 \Sigma\succeq(1-\gamma^2/n)I_d\succeq\tfrac12 I_d,
 \qquad
 \Pi_k\succeq\frac{\alpha m_k}{2n}I_d\succ0.
\]
Apply \cref{thm:2} with
\[
        M=0,
        \qquad
        \Delta_0=\gamma^2I_d/n,
        \qquad
        H=(1+\gamma^2/n)I_d-\Sigma.
\]
Since \(0\preceq H\preceq2\gamma^2I_d/n\), the two Gaussian-augmentation
terms in \cref{thm:2} are bounded by \(C\gamma\sqrt p/\sqrt n\).
As \(|\xi_k|\le K/\sqrt n\), \cref{thm:2} and the Gaussian moment bound give
\[
\begin{aligned}
\mcl W_{p,2}\bigl(\mcl L(S_n),N(0,I_d)\bigr)
&\le
\frac{Cp}{\sqrt n}\sum_{k=1}^n\frac1{m_k}
+\frac{C\gamma\sqrt p}{\sqrt n} \\
&\le Cp\,\delta,
\qquad p\ge1,
\end{aligned}
\]
where \(C\) depends only on \(K,\alpha,\gamma\), and \(d\). For
\(1\le n<2\gamma^2\), the bound
\(\mcl W_{p,2}\le K\sqrt n+\|Z\|_{p,2}\) gives the same estimate after
increasing \(C\). For all sufficiently large \(n\), the desired result
\cref{eq:thmMD} follows from \cite[Theorem~4.2]{fang2023p}. The finitely
many remaining \(n\) are covered by increasing \(C\), since
\(\mathbb P(|Z|>x)\) is bounded away from zero on their prescribed
ranges \(0\le x\le\delta^{-1/3}\).
\end{proof}

To prepare for the proof of \cref{prop:SGD}, we need the following lemma obtained by a standard computation.

\begin{lemma}\label{lem:sgd-auxiliary}
Under the assumptions of \cref{prop:SGD}, put
\[
 A_n=\sum_{i=0}^{n-1}\prod_{j=1}^{i}(I_d-\eta_jH),
\]
and, for $1\leq k\leq n-1$, define
\begin{equation}\label{eq:sgd-weight}
 B_k^n=\eta_k\sum_{i=k}^{n-1}
 \prod_{\ell=k+1}^{i}(I_d-\eta_\ell H).
\end{equation}
Empty products equal $I_d$. There is a constant $C$, with the same allowed
dependence as in \cref{prop:SGD}, such that
\begin{gather}
 \sup_{n\geq1}\|A_n\|_2\leq C,
 \qquad
 \|\Delta_k\|_{2p}\leq C(k+1)^{-\alpha/2},\quad k\geq0,
 \label{eq:sgd-auxiliary-stability}\\
 \max_{1\leq k<n}\|B_k^n\|_2\leq C,
 \qquad
 \sum_{k=1}^{n-1}\|B_k^n-H^{-1}\|_2\leq Cn^\alpha,
 \quad n\geq2.
 \label{eq:sgd-weight-defect}
\end{gather}
\end{lemma}

\begin{proof}
Since $\eta_0L\leq1/2$,
\[
 \left\|\prod_{j=1}^{i}(I_d-\eta_jH)\right\|_2
 \leq\exp\left(-\mu\sum_{j=1}^{i}\eta_j\right)
 \leq C\exp\{-c(i+1)^{1-\alpha}\},\qquad i\geq0.
\]
Summing over $i$ proves the bound on $A_n$.

For the iterate moment, put $q=2p$ and $m_k=\mathbb E|\Delta_k|_2^q$.
By \cref{eq:sgd-regularity},
\[
 |u-\eta_k\nabla f(\theta^*+u)|_2
 \leq(1-\mu\eta_k)|u|_2.
\]
For $q>2$, the Taylor bound for $x\mapsto|x|_2^q$ is
\[
\begin{aligned}
 |x+y|_2^q
 &\leq |x|_2^q+q|x|_2^{q-2}\langle x,y\rangle\\
 &\quad+C_q\bigl(|x|_2^{q-2}|y|_2^2+|y|_2^q\bigr),
 \qquad x,y\in\mathbb R^d.
\end{aligned}
\]
Apply it with $x=\Delta_{k-1}-\eta_k\nabla f(\theta_{k-1})$ and
$y=-\eta_k\varepsilon_k$. Since $x$ is $\mathcal F_{k-1}$-measurable,
\[
 \mathbb E\bigl[|x|_2^{q-2}\langle x,y\rangle\mid\mathcal F_{k-1}\bigr]
 =|x|_2^{q-2}\langle x,\mathbb E(y\mid\mathcal F_{k-1})\rangle=0.
\]
Thus, using \cref{eq:sgd-moment},
\[
 m_k\leq(1-\mu\eta_k)^q m_{k-1}
 +C\eta_k^2\mathbb E\bigl(|\Delta_{k-1}|_2^{q-2}|\varepsilon_k|_2^2\bigr)
 +C\eta_k^q.
\]
H\"older's inequality with exponents $q/(q-2)$ and $q/2$ gives
\[
 \mathbb E\bigl(|\Delta_{k-1}|_2^{q-2}|\varepsilon_k|_2^2\bigr)
 \leq m_{k-1}^{1-2/q}\|\varepsilon_k\|_{q,2}^2
 \leq C m_{k-1}^{1-2/q}.
\]
Young's inequality with the same exponents then yields
\[
\begin{aligned}
 C\eta_k^2m_{k-1}^{1-2/q}
 &=C(\eta_km_{k-1})^{1-2/q}\eta_k^{1+2/q}\\
 &\leq\frac\mu2\eta_km_{k-1}+C\eta_k^{q/2+1}.
\end{aligned}
\]
For $q=2$, expansion of the square gives
\[
 m_k\leq(1-\mu\eta_k)^2m_{k-1}+K_{3p}^2\eta_k^2.
\]
Using $(1-\mu\eta_k)^q\leq1-\mu\eta_k$ and
$\eta_k^q\leq C\eta_k^{q/2+1}$, we obtain in both cases
\[
 m_k\leq(1-c\eta_k)m_{k-1}+C\eta_k^{q/2+1}
 \quad\Longrightarrow\quad
 m_k\leq C(k+1)^{-\alpha q/2}.
\]
A standard induction, using $k^{-1}=o(\eta_k)$, yields
\[
m_k \le C(k+1)^{-\alpha q/2}.
\]
This proves \cref{eq:sgd-auxiliary-stability}.

For the weight estimates, put
\[
 P_{k,i}=\prod_{\ell=k+1}^{i}(I_d-\eta_\ell H),
 \qquad k\leq i\leq n.
\]
Since $P_{k,k}=I_d$ and all factors commute with $H$,
\[
 \eta_{i+1}P_{k,i}
 =H^{-1}(P_{k,i}-P_{k,i+1}).
\]
Summing this identity gives
\[
 \sum_{i=k}^{n-1}\eta_{i+1}P_{k,i}
 =H^{-1}\sum_{i=k}^{n-1}(P_{k,i}-P_{k,i+1})
 =H^{-1}(I_d-P_{k,n}).
\]
Thus, by \cref{eq:sgd-weight},
\[
 B_k^n-H^{-1}
 =-H^{-1}P_{k,n}
 +\sum_{i=k}^{n-1}(\eta_k-\eta_{i+1})P_{k,i}.
\]

To estimate the product, use $\eta_\ell L\leq1/2$ to obtain
\[
\begin{aligned}
 \|P_{k,i}\|_2
 &\leq\prod_{\ell=k+1}^{i}(1-\mu\eta_\ell)
 \leq\exp\left(
 -\mu\eta_0\sum_{\ell=k+1}^{i}\ell^{-\alpha}
 \right)\\
 &\leq e^{-\mu\eta_0(i-k)i^{-\alpha}}
 \leq
 \begin{cases}
  e^{-c(i-k)k^{-\alpha}},
  &k\leq i\leq\min(2k,n),\\
  e^{-ci^{1-\alpha}},
  &2k<i\leq n.
 \end{cases}
\end{aligned}
\]
The two cases use $i\leq2k$ and $i-k\geq i/2$, respectively. Also,
\[
\begin{gathered}
 0\leq\eta_k-\eta_{i+1}
 =\alpha\eta_0\int_k^{i+1}x^{-\alpha-1}\,dx
 \leq C(i-k+1)k^{-\alpha-1},\\
 \eta_k-\eta_{i+1}\leq\eta_k=\eta_0k^{-\alpha}.
\end{gathered}
\]

For the two sums below, we use
\[
\begin{aligned}
 \sum_{s=0}^{k}(s+1)e^{-csk^{-\alpha}}
 &\leq\sum_{s=0}^{\infty}(s+1)e^{-csk^{-\alpha}}
 =\frac1{(1-e^{-ck^{-\alpha}})^2}
 \leq Ck^{2\alpha},\\
 \sum_{i>2k}e^{-ci^{1-\alpha}}
 &\leq\sum_{i=1}^{\infty}e^{-ci^{1-\alpha}}<\infty.
\end{aligned}
\]
Here $1-e^{-ck^{-\alpha}}\geq(1-e^{-c})k^{-\alpha}$,
and $1-\alpha>0$.
Splitting at $i=2k$ and writing $s=i-k$ in the first part therefore gives
\[
\begin{aligned}
 \sum_{i=k}^{n-1}(\eta_k-\eta_{i+1})\|P_{k,i}\|_2
 &\leq Ck^{-\alpha-1}
 \sum_{s=0}^{k}(s+1)e^{-csk^{-\alpha}}\\
 &\quad+Ck^{-\alpha}\sum_{i>2k}e^{-ci^{1-\alpha}}\\
 &\leq Ck^{\alpha-1}+Ck^{-\alpha}
 \leq Ck^{\alpha-1},
\end{aligned}
\]
since $\alpha>1/2$.
Taking $i=n$ in the product estimate and using the weight identity,
we obtain
\begin{equation}\label{eq:sgd-W-pointwise}
 \|B_k^n-H^{-1}\|_2
 \leq C\{k^{\alpha-1}+e^{-c(n-k)n^{-\alpha}}\}.
\end{equation}

Both terms in braces are at most one, so $\|B_k^n\|_2$
is uniformly bounded.
Finally, summing \cref{eq:sgd-W-pointwise} and setting $j=n-k$ gives
\[
\begin{aligned}
 \sum_{k=1}^{n-1}\|B_k^n-H^{-1}\|_2
 &\leq C\sum_{k=1}^{n-1}k^{\alpha-1}
      +C\sum_{j=1}^{n-1}e^{-cjn^{-\alpha}}\\
 &\leq C\left(1+\int_1^n x^{\alpha-1}\,dx\right)
      +\frac{C}{e^{cn^{-\alpha}}-1}\\
 &\leq Cn^\alpha,
\end{aligned}
\]
where we used the geometric series and $e^x-1\geq x$ for $x\geq0$.
This proves \cref{eq:sgd-weight-defect}.
\end{proof}

\begin{proof}[Proof of \cref{prop:SGD}]
Put $T_n=-n^{-1/2}\sum_{k=1}^{n-1}B_k^n\varepsilon_k$.
Iterating the recursion and averaging yields
\begin{equation}\label{eq:sgd-decomposition}
 \sqrt n(\bar\theta_n-\theta^*)-T_n
 =\frac{A_n\Delta_0}{\sqrt n}
 -\frac1{\sqrt n}\sum_{k=1}^{n-1}
 B_k^n\{\nabla f(\theta_{k-1})-H\Delta_{k-1}\}.
\end{equation}
Since $\nabla f(\theta^*)=0$, \cref{eq:sgd-regularity} gives
\[
\begin{aligned}
 \nabla f(\theta^*+u)-Hu
 &=\int_0^1\{\nabla^2f(\theta^*+su)-H\}u\,ds,\\
 |\nabla f(\theta^*+u)-Hu|_2
 &\leq\int_0^1L_Hs|u|_2^2\,ds=\frac{L_H}{2}|u|_2^2.
\end{aligned}
\]
Thus \cref{lem:sgd-auxiliary} gives
\begin{equation}\label{eq:sgd-remainder}
\begin{aligned}
 \|\sqrt n(\bar\theta_n-\theta^*)-T_n\|_{p,2}
 &\leq\frac{\|A_n\|_2\|\Delta_0\|_{p,2}}{\sqrt n}
 +\frac{L_H}{2\sqrt n}\sum_{k=1}^{n-1}
 \|B_k^n\|_2\|\Delta_{k-1}\|_{2p,2}^2\\
 &\leq C\{n^{-1/2}+n^{1/2-\alpha}\}.
\end{aligned}
\end{equation}

For the martingale term $T_n$, define the deterministic covariances
\begin{equation}\label{eq:sgd-finite-covariance}
 \Pi_{k,n}^B=\frac1n\sum_{j=k}^{n-1}B_j^nV(B_j^n)^\top,
 \qquad
 \Gamma_n^{\mathrm{PR}}=\Pi_{1,n}^B=\operatorname{Cov}(T_n),
\end{equation}
for $1\leq k<n$, with $\Pi_{n,n}^B=0$.
Condition \cref{eq:sgd-fixed-covariance} makes these the predictable
covariance tails. Since $B_{n-1}^n=\eta_{n-1}I_d$ and
$\|V\|_2\leq K_{3p}^2$,
\[
 \Pi_{k,n}^B\succeq\Pi_{n-1,n}^B
 =\frac{\eta_{n-1}^2}{n}V\succ0,
 \qquad
 \Gamma_n^{\mathrm{PR}}\preceq CI_d.
\]
Also,
\[
 \left\|\|\varepsilon_k\mid\mathcal F_{k-1}\|_{3p,2}^{3}\right\|_p
 =\|\varepsilon_k\|_{3p,2}^{3}\leq K_{3p}^{3}.
\]
Applying \cref{thm:3} to $-n^{-1/2}B_k^n\varepsilon_k$ with $M=0$,
\[
\begin{aligned}
 &\mcl W_{p,2}\bigl(\mcl L(T_n),\N(0,\Gamma_n^{\mathrm{PR}})\bigr)\\
 &\quad\leq\frac{Cp}{n^{3/2}}
 \sum_{k=1}^{n-1}\|B_k^n\|_2
 \|(\Pi_{k,n}^B)^{-1/2}B_k^n\|_2^2
 \|\varepsilon_k\|_{3p,2}^3.
\end{aligned}
\]
Since $I_d\preceq\|V^{-1}\|_2V$, the operator norm--trace comparison gives
\[
\begin{aligned}
 \frac1n\|(\Pi_{k,n}^B)^{-1/2}B_k^n\|_2^2
 &\leq\frac{\|V^{-1}\|_2}{n}
 \operatorname{tr}\bigl[(\Pi_{k,n}^B)^{-1}B_k^nV(B_k^n)^\top\bigr]\\
 &=\|V^{-1}\|_2\operatorname{tr}\bigl[
 (\Pi_{k,n}^B)^{-1}(\Pi_{k,n}^B-\Pi_{k+1,n}^B)\bigr]
\end{aligned}
\]
Hence
\[
 \mcl W_{p,2}\bigl(\mcl L(T_n),\N(0,\Gamma_n^{\mathrm{PR}})\bigr)
 \leq\frac C{\sqrt n}\sum_{k=1}^{n-1}
 \operatorname{tr}\bigl[(\Pi_{k,n}^B)^{-1}
 (\Pi_{k,n}^B-\Pi_{k+1,n}^B)\bigr].
\]
Concavity of $\log\det$ gives
\[
 \operatorname{tr}\{A^{-1}(A-B)\}
 \leq\log\det A-\log\det B,
 \qquad A\succeq B\succ0.
\]
The last summand below equals $d$; the others telescope:
\[
\begin{aligned}
 \sum_{k=1}^{n-1}\operatorname{tr}\bigl[
 (\Pi_{k,n}^B)^{-1}(\Pi_{k,n}^B-\Pi_{k+1,n}^B)\bigr]
 &\leq d+\sum_{k=1}^{n-2}\log
 \frac{\det\Pi_{k,n}^B}{\det\Pi_{k+1,n}^B}\\
 &=d+\log\frac{\det\Gamma_n^{\mathrm{PR}}}
 {\det\{\eta_{n-1}^2V/n\}}\\
 &\leq C+d\log n+2\alpha d\log(n-1)\leq C\log(en),
\end{aligned}
\]
where $\Gamma_n^{\mathrm{PR}}\preceq CI_d$ and
$\det\{\eta_{n-1}^2V/n\}=\eta_0^{2d}n^{-d}(n-1)^{-2\alpha d}\det V$.
Consequently,
\begin{equation}\label{eq:sgd-weighted-martingale}
 \mcl W_{p,2}\bigl(\mcl L(T_n),\N(0,\Gamma_n^{\mathrm{PR}})\bigr)
 \leq C\frac{\log(en)}{\sqrt n}.
\end{equation}

It remains to compare the two Gaussian covariances. By
\cref{eq:sgd-weight-defect},
\begin{equation}\label{eq:sgd-covariance-comparison}
\begin{aligned}
 \|\Gamma_n^{\mathrm{PR}}-H^{-1}VH^{-1}\|_2
 &\leq\frac{\|H^{-1}VH^{-1}\|_2}{n}\\
 &\quad+\frac{\|V\|_2}{n}\sum_{k=1}^{n-1}
 (\|B_k^n\|_2+\|H^{-1}\|_2)\|B_k^n-H^{-1}\|_2\\
 &\leq Cn^{\alpha-1}.
\end{aligned}
\end{equation}
For positive definite $A,B$, the Sylvester equation implies
\[
 \|A^{1/2}-B^{1/2}\|_2
 \leq\frac{\|A-B\|_2}
 {\sqrt{\lambda_{\min}(A)}+\sqrt{\lambda_{\min}(B)}}.
\]
Since $H^{-1}VH^{-1}\succeq
 (L^2\|V^{-1}\|_2)^{-1}I_d$, coupling by the same standard Gaussian gives
\[
 \mcl W_{p,2}\bigl(\N(0,\Gamma_n^{\mathrm{PR}}),\N(0,H^{-1}VH^{-1})\bigr)
 \leq C\|\Gamma_n^{\mathrm{PR}}-H^{-1}VH^{-1}\|_2
 \leq Cn^{\alpha-1}.
\]
Combining this with \cref{eq:sgd-remainder,eq:sgd-weighted-martingale}
and the triangle inequality proves \cref{eq:sgd-limit-bound}.
\end{proof}


\bibliographystyle{apalike} 
\bibliography{reference}

@inproceedings{anastasiou2019normal,
  title={Normal approximation for stochastic gradient descent via non-asymptotic rates of martingale {CLT}},
  author={Anastasiou, Andreas and Balasubramanian, Krishnakumar and Erdogdu, Murat A.},
  booktitle={Proceedings of the Thirty-Second Conference on Learning Theory},
  pages={115--137},
  year={2019},
  editor={Beygelzimer, Alina and Hsu, Daniel},
  volume={99},
  series={Proceedings of Machine Learning Research},
  publisher={PMLR}
}

@article{bonis2020stein,
  title={Stein’s method for normal approximation in {Wasserstein} distances with application to the multivariate central limit theorem},
  author={Bonis, Thomas},
  journal={Probability Theory and Related Fields},
  volume={178},
  number={3},
  pages={827--860},
  year={2020},
  publisher={Springer}
}

@book{boucheron2013,
  title={Concentration Inequalities: A Nonasymptotic Theory of Independence},
  author={Boucheron, St\'ephane and Massart, Pascal and Lugosi G\'abor},
  year={2013},
  publisher={Oxford University Press}
}

@article{cattaneo2025yurinskii,
  title={Yurinskii’s coupling for martingales},
  author={Cattaneo, Matias D and Masini, Ricardo P and Underwood, William G},
  journal={The Annals of Statistics},
  volume={53},
  number={5},
  pages={2179--2203},
  year={2025},
  publisher={Institute of Mathematical Statistics}
}

@article{Cramer1938,
  title={Sur un nouveau th{\'e}oreme-limite de la th{\'e}orie des probabilit{\'e}s},
  author={Cram{\'e}r, Harald},
  journal={Actual. Sci. Ind.},
  volume={736},
  pages={5--23},
  year={1938}
}

@article{fan2024cramer,
  title={Cram{\'e}r’s moderate deviations for martingales with applications},
  author={Fan, Xiequan and Shao, Qi-Man},
  journal={Annales de l'Institut Henri Poincare (B) Probabilites et statistiques},
  volume={60},
  number={3},
  pages={2046--2074},
  year={2024},
  organization={Institut Henri Poincar{\'e}}
}

@article{fang2024sharp,
  title={Sharp high-dimensional central limit theorems for log-concave distributions},
  author={Fang, Xiao and Koike, Yuta},
  journal={Annales de l'Institut Henri Poincare (B) Probabilites et statistiques},
  volume={60},
  number={3},
  pages={2129--2156},
  year={2024},
  organization={Institut Henri Poincar{\'e}}
}

@article{fang2023p,
  title={From $p$-{Wasserstein} bounds to moderate deviations},
  author={Fang, Xiao and Koike, Yuta},
  journal={Electronic Journal of Probability},
  volume={28},
  pages={1--52},
  year={2023},
  publisher={The Institute of Mathematical Statistics and the Bernoulli Society}
}

@article{kong2026finite,
  title={Finite-Sample {Wasserstein} Error Bounds and Concentration Inequalities for Nonlinear Stochastic Approximation},
  author={Kong, Seo Taek and Srikant, R},
  journal={arXiv preprint arXiv:2602.02445},
  year={2026}
}

@article{koike2026note,
  title={Connections between the {F}{\"o}llmer process and the denoising diffusion probabilistic model},
  author={Koike, Yuta},
  journal={arXiv preprint arXiv:2605.18040},
  year={2026}
}

@article{LeNoPe15,
  title={Stein’s method, logarithmic {S}obolev and transport inequalities},
  author={Ledoux, Michel and Nourdin, Ivan and Peccati, Giovanni},
  journal={Geometric and Functional Analysis},
  volume={25},
  number={1},
  pages={256--306},
  year={2015},
  publisher={Springer}
}

@article{otto2000generalization,
  title={Generalization of an inequality by {Talagrand} and links with the logarithmic {Sobolev} inequality},
  author={Otto, Felix and Villani, C{\'e}dric},
  journal={Journal of Functional Analysis},
  volume={173},
  number={2},
  pages={361--400},
  year={2000},
  publisher={Elsevier}
}

@article{pinelis1994optimum,
  title={Optimum bounds for the distributions of martingales in {Banach} spaces},
  author={Pinelis, Iosif},
  journal={The Annals of Probability},
  pages={1679--1706},
  year={1994},
  publisher={JSTOR}
}

@article{rollin2018quantitative,
  title={On quantitative bounds in the mean martingale central limit theorem},
  author={R{\"o}llin, Adrian},
  journal={Statistics \& Probability Letters},
  volume={138},
  pages={171--176},
  year={2018},
  publisher={Elsevier}
}

@article{shao2022berry,
  title={Berry--Esseen bounds for multivariate nonlinear statistics with applications to {M}-estimators and stochastic gradient descent algorithms},
  author={Shao, Qi-Man and Zhang, Zhuo-Song},
  journal={Bernoulli},
  volume={28},
  number={3},
  pages={1548--1576},
  year={2022},
  doi={10.3150/21-BEJ1336}
}

@inproceedings{sheshukova2026gaussian,
  title={Gaussian approximation and multiplier bootstrap for stochastic gradient descent},
  author={Sheshukova, Marina and Samsonov, Sergey and Belomestny, Denis and Moulines, Eric and Shao, Qi-Man and Zhang, Zhuo-Song and Naumov, Alexey},
  booktitle={Proceedings of the 29th International Conference on Artificial Intelligence and Statistics},
  pages={1378--1386},
  year={2026},
  volume={300},
  series={Proceedings of Machine Learning Research},
  publisher={PMLR}
}

@article{srikant2025rates,
  title={Rates of convergence in the central limit theorem for {Markov} chains, with an application to {TD} learning},
  author={Srikant, R.},
  journal={Mathematics of Operations Research},
  pages={1--18},
  year={2025},
  doi={10.1287/moor.2024.0444}
}

@article{wu2025uncertainty,
  title={Uncertainty quantification for {Markov} chain induced martingales with application to temporal difference learning},
  author={Wu, Weichen and Wei, Yuting and Rinaldo, Alessandro},
  journal={arXiv preprint arXiv:2502.13822v3},
  year={2025}
}

@book{hall1980martingale,
  author    = {Hall, Peter and Heyde, Christopher C.},
  title     = {Martingale Limit Theory and Its Application},
  publisher = {Academic Press},
  address   = {New York},
  year      = {1980}
}

@article{li2018applications,
  author  = {Li, Danning and Xue, Lingzhou and Zou, Hui},
  title   = {Applications of {Peter Hall}'s Martingale Limit Theory to Estimating and Testing High Dimensional Covariance Matrices},
  journal = {Statistica Sinica},
  volume  = {28},
  number  = {4},
  pages   = {2657--2670},
  year    = {2018},
  doi     = {10.5705/ss.202017.0060}
}

@article{polyak1992acceleration,
  author  = {Polyak, Boris T. and Juditsky, Anatoli B.},
  title   = {Acceleration of Stochastic Approximation by Averaging},
  journal = {SIAM Journal on Control and Optimization},
  volume  = {30},
  number  = {4},
  pages   = {838--855},
  year    = {1992},
  doi     = {10.1137/0330046}
}

@misc{borkar2021ode,
  author        = {Borkar, Vivek S. and Chen, Shuhang and Devraj, Adithya and Kontoyiannis, Ioannis and Meyn, Sean P.},
  title         = {The {ODE} Method for Asymptotic Statistics in Stochastic Approximation and Reinforcement Learning},
  journal={arXiv preprint arXiv:2110.14427},
  year          = {2021}
}

@inproceedings{hu2024central,
  author    = {Hu, Jie and Doshi, Vishwaraj and Eun, Do Young},
  title     = {Central Limit Theorem for Two-Timescale Stochastic Approximation with {Markovian} Noise: Theory and Applications},
  booktitle = {Proceedings of the 27th International Conference on Artificial Intelligence and Statistics},
  series    = {Proceedings of Machine Learning Research},
  volume    = {238},
  publisher = {PMLR},
  year      = {2024}
}

@article{bolthausen1982exact,
  author  = {Bolthausen, Erwin},
  title   = {Exact Convergence Rates in Some Martingale Central Limit Theorems},
  journal = {The Annals of Probability},
  volume  = {10},
  number  = {3},
  pages   = {672--688},
  year    = {1982},
  doi     = {10.1214/aop/1176993776}
}

@article{haeusler1988nonuniform,
  author  = {Haeusler, Erich and Joos, Konrad},
  title   = {A Nonuniform Bound on the Rate of Convergence in the Martingale Central Limit Theorem},
  journal = {The Annals of Probability},
  volume  = {16},
  number  = {4},
  pages   = {1699--1720},
  year    = {1988},
  doi     = {10.1214/aop/1176991592}
}

@article{mourrat2013rate,
  author  = {Mourrat, Jean-Christophe},
  title   = {On the Rate of Convergence in the Martingale Central Limit Theorem},
  journal = {Bernoulli},
  volume  = {19},
  number  = {2},
  pages   = {633--645},
  year    = {2013},
  doi     = {10.3150/12-BEJ417}
}

@article{fan2020wasserstein,
  author  = {Fan, Xiequan and Ma, Xiaohui},
  title   = {On the {Wasserstein} Distance for a Martingale Central Limit Theorem},
  journal = {Statistics \& Probability Letters},
  volume  = {167},
  pages   = {108892},
  year    = {2020},
  doi     = {10.1016/j.spl.2020.108892}
}

@article{dedecker2022rates,
  author  = {Dedecker, J{\'e}r{\^o}me and Merlev{\`e}de, Florence and Rio, Emmanuel},
  title   = {Rates of Convergence in the Central Limit Theorem for Martingales in the Non-stationary Setting},
  journal = {Annales de l'Institut Henri Poincar{\'e}, Probabilit{\'e}s et Statistiques},
  volume  = {58},
  number  = {2},
  pages   = {945--966},
  year    = {2022},
  doi     = {10.1214/21-AIHP1182}
}

@article{guo2024wasserstein,
  author  = {Guo, Xiaoqin},
  title   = {On the Rate of Convergence of the Martingale Central Limit Theorem in {Wasserstein} Distances},
  journal = {Electronic Journal of Probability},
  volume  = {31},
  number  = {68},
  pages   = {1--29},
  year    = {2026},
  doi     = {10.1214/26-EJP1519}
}

@misc{belloni2018high,
  author        = {Belloni, Alexandre and Oliveira, Roberto I.},
  title         = {A High Dimensional Central Limit Theorem for Martingales, with Applications to Context Tree Models},
  journal={arXiv preprint arXiv:1809.02741},
  year          = {2018}
}

@misc{wu2026berry,
  author        = {Wu, Weichen and Le, Dung and Kuchibhotla, Arun Kumar and Rinaldo, Alessandro},
  title         = {{Berry--Esseen} Bounds for Multivariate Martingale Difference Sequences in the {Kolmogorov} Distance},
  journal={arXiv preprint arXiv:2605.03100},
  year          = {2026}
}

@article{yurinskii1978error,
  title={On the error of the {Gaussian} approximation for convolutions},
  author={Yurinskii, Vadim V},
  journal={Theory of Probability \& its Applications},
  volume={22},
  number={2},
  pages={236--247},
  year={1978},
  publisher={SIAM}
}

@article{paulin2026theoretical,
  author        = {Paulin, Daniel and Whalley, Peter A.},
  title         = {Theoretical Guarantees for Stochastic Gradient Sampling Methods via {Gaussian} Convolution Inequalities},
  journal={arXiv preprint arXiv:2604.24632},
  year          = {2026}
}

@article{li2020uniform,
  title={Uniform nonparametric inference for time series},
  author={Li, Jia and Liao, Zhipeng},
  journal={Journal of Econometrics},
  volume={219},
  number={1},
  pages={38--51},
  year={2020},
  publisher={Elsevier}
}

@article{von1967multi,
  title={Multi-dimensional integral limit theorems for large deviations},
  author={von Bahr, Bengt},
  journal={Arkiv f{\"o}r matematik},
  volume={7},
  number={1},
  pages={89--99},
  year={1967},
  publisher={Kluwer Academic Publishers Dordrecht}
}

@book{petrov2012sums,
  title={Sums of independent random variables},
  author={Petrov, Valentin V},
  year={2012},
  publisher={Springer Science \& Business Media}
}

@article{zhai2018high,
  title={A high-dimensional {CLT} in {$\mathcal{W}_2$} distance with near optimal convergence rate},
  author={Zhai, Alex},
  journal={Probability Theory and Related Fields},
  volume={170},
  number={3--4},
  pages={821--845},
  year={2018},
  doi={10.1007/s00440-017-0771-3}
}

@article{zhang2026gaussian,
  title={Gaussian approximation for multivariate martingale sums from uniformly ergodic {Markov} chains},
  author={Zhang, Yixuan and Xie, Qiaomin},
  journal={arXiv preprint arXiv:2609.09480},
  year={2026}
}

@article{bolbotowski2026sharp,
  title={Sharp inequalities between Zolotarev and Wasserstein distances in ${P}_2(\mathbb{R}^d)$},
  author={Bo{\l}botowski, Karol and Bouchitt{\'e}, Guy},
  journal={Probability Theory and Related Fields},
  pages={1--17},
  year={2026},
  publisher={Springer}
}









\end{document}